\documentclass{amsart} 
\usepackage{amssymb}
\usepackage{amsmath}
\usepackage{amsfonts}
\usepackage{fix-cm}    

\makeatletter
\newcommand\HUGE{\@setfontsize\Huge{30}{36}}
\makeatother

\makeatletter
\def\bbordermatrix#1{\begingroup \m@th
  \@tempdima 4.75\p@
  \setbox\z@\vbox{%
    \def\cr{\crcr\noalign{\kern2\p@\global\let\cr\endline}}%
    \ialign{$##$\hfil\kern2\p@\kern\@tempdima&\thinspace\hfil$##$\hfil
      &&\quad\hfil$##$\hfil\crcr
      \omit\strut\hfil\crcr\noalign{\kern-\baselineskip}%
      #1\crcr\omit\strut\cr}}%
  \setbox\tw@\vbox{\unvcopy\z@\global\setbox\@ne\lastbox}%
  \setbox\tw@\hbox{\unhbox\@ne\unskip\global\setbox\@ne\lastbox}%
  \setbox\tw@\hbox{$\kern\wd\@ne\kern-\@tempdima\left[\kern-\wd\@ne
    \global\setbox\@ne\vbox{\box\@ne\kern2\p@}%
    \vcenter{\kern-\ht\@ne\unvbox\z@\kern-\baselineskip}\,\right]$}%
  \null\;\vbox{\kern\ht\@ne\box\tw@}\endgroup}
\makeatother

\begin{document}
\newtheorem{theo}{Theorem}[section]
\newtheorem{prop}[theo]{Proposition}
\newtheorem{lemma}[theo]{Lemma}
\newtheorem{exam}[theo]{Example}
\newtheorem{coro}[theo]{Corollary}
\theoremstyle{definition}
\newtheorem{defi}[theo]{Definition}
\newtheorem{rem}[theo]{Remark}


\newcommand{\Bb}{{\bf B}}
\newcommand{\Nb}{{\bf N}}
\newcommand{\Qb}{{\bf Q}}
\newcommand{\Rb}{{\bf R}}
\newcommand{\Zb}{{\bf Z}}
\newcommand{\Ac}{{\mathcal A}}
\newcommand{\Bc}{{\mathcal B}}
\newcommand{\Cc}{{\mathcal C}}
\newcommand{\Dc}{{\mathcal D}}
\newcommand{\Fc}{{\mathcal F}}
\newcommand{\Ic}{{\mathcal I}}
\newcommand{\Jc}{{\mathcal J}}
\newcommand{\Lc}{{\mathcal L}}
\newcommand{\Oc}{{\mathcal O}}
\newcommand{\Pc}{{\mathcal P}}
\newcommand{\Sc}{{\mathcal S}}
\newcommand{\Tc}{{\mathcal T}}
\newcommand{\Uc}{{\mathcal U}}
\newcommand{\Vc}{{\mathcal V}}

\newcommand{\ax}{{\rm ax}}
\newcommand{\Acc}{{\rm Acc}}
\newcommand{\Act}{{\rm Act}}
\newcommand{\ded}{{\rm ded}}
\newcommand{\Gm}{{$\Gamma_0$}}
\newcommand{\ID}{{${\rm ID}_1^i(\Oc)$}}
\newcommand{\PA}{{\rm PA}}
\newcommand{\ACA}{{${\rm ACA}^i$}}
\newcommand{\RefP}{{${\rm Ref}^*({\rm PA}(P))$}}
\newcommand{\RefS}{{${\rm Ref}^*({\rm S}(P))$}}
\newcommand{\Rfn}{{\rm Rfn}}
\newcommand{\tar}{{\rm Tarski}}
\newcommand{\UNFA}{{${\mathcal U}({\rm NFA})$}}

\author{Nik Weaver}

\title [Hereditarily antisymmetric operator algebras]
       {Hereditarily antisymmetric operator algebras}

\address {Department of Mathematics\\
          Washington University in Saint Louis\\
          Saint Louis, MO 63130}

\email {nweaver@math.wustl.edu}

\date{\em April 20, 2019}

\begin{abstract}
We introduce a notion of ``hereditarily antisymmetric'' operator algebras
and prove a structure theorem for them in finite dimensions. We also
characterize those operator algebras in finite dimensions which can be made
upper triangular and prove matrix analogs of the theorems of Dilworth and
Mirsky for finite posets. Some partial results are obtained in the infinite
dimensional case.
\end{abstract}

\maketitle


In this paper an {\it operator algebra} is a linear subspace of some
$\mathcal{B}(\mathcal{H})$ which is stable under products, where $\mathcal{H}$
is a complex Hilbert space. It is {\it unital} if it contains the identity
operator $I$. Self-adjointness is not assumed. Indeed, an operator algebra
$\mathcal{A}$ is said to be {\it antisymmetric} if $\mathcal{A} \cap
\mathcal{A}^*$ is $\{0\}$ or $\mathbb{C}\cdot I$. This means that a nonunital
antisymmetric operator algebra contains no nonzero self-adjoint operators, and
the only self-adjoint operators a unital one contains are the real multiples
of the identity operator (Proposition \ref{sainf}). Antisymmetric operator
algebras were introduced in \cite{Sz} and have attracted occasional attention,
e.g., \cite{KP1, KP2}.

Following \cite{W1}, we regard linear subspaces $\mathcal{V}$ of
$M_n = M_n(\mathbb{C}) \cong \mathcal{B}(\mathbb{C}^n)$ as matrix or
``quantum'' analogs of relations on finite sets. In this picture the quantum
analog of a preorder relation, i.e., a relation which is reflexive and
transitive, is a
unital operator algebra. This intuition behind this idea, and its relation to
the physics of finite state quantum systems, is discussed below in Section 1.

Partially ordered sets are preordered sets which satisfy the extra
condition of antisymmetry ($a \leq b$ and $b \leq a$ implies $a = b$).
On the face of it, a natural matrix version of this condition might be
for the operator algebra to be antisymmetric in the sense defined above.
This suggests that antisymmetric operator algebras should be rather special
compared to general operator algebras, in something like the way that posets
are special compared to general preorders. But that does not
seem to be the case. Of course, just by counting dimensions, it is easy to
see that an operator algebra in $M_n$ cannot be antisymmetric if its
dimension is at least $\frac{n^2}{2} + 1$ (Proposition \ref{dimcount}). But
operator algebras of lower dimension than this typically should not
nontrivially intersect the real linear space of self-adjoint matrices. In
other words, the operator algebras of dimension less than $\frac{n^2}{2} + 1$
which are {\it not} antisymmetric are exceptional, not the other way around.

In contrast, I will show in this paper that operator algebras which enjoy a
sort of ``hereditary'' antisymmetry condition are unusually well-behaved. In
particular, the main result, Theorem \ref{mainthm}, establishes that in
finite dimensions they must have a very special structure. A key result
on the way, Theorem \ref{utthm}, characterizes those operator algebras in
$M_n$ which can be put in upper triangular form. I propose that hereditarily
antisymmetric operator algebras are the right notion of ``quantum poset''.

One aspect of the special structure of these algebras is that they can always
be decomposed into a diagonal subalgebra (relative to a suitable basis) and a
nilpotent ideal (Corollary \ref{bigdiag} plus Theorem \ref{mainthm}). The
intuition for this decomposition could be that
the diagonal subalgebra represents
the ``equality'' part of the quantum partial order and the nilpotent ideal
represents the ``strict inequality'' part. Moreover, every nilpotent operator
algebra is hereditarily antisymmetric (Theorem \ref{tfaeex1} (v) $\Rightarrow$
(i) plus Corollary \ref{tfaecor}), so we are led to the view that nilpotent
operator algebras are ``quantum'' strict orders. In support of this idea,
I prove matrix analogs of two basic theorems about finite posets,
Dilworth's theorem and Mirsky's theorem, for nilpotent operator algebras in
$M_n$ (Theorems \ref{qmir} and \ref{qdil}). The quantum Mirsky theorem has
essentially the same proof as its classical analog, while the quantum
Dilworth theorem has a very nonclassical proof.

Trivially, any subalgebra of an antisymmetric operator algebra must itself
be antisymmetric. So requiring $\mathcal{A}$ and all of its subalgebras to be
antisymmetric is no different from merely requiring $\mathcal{A}$ to be
antisymmetric. Whereas imposing this condition on quotients of $\mathcal{A}$
does not even make sense, as the concept of antisymmetry depends on the
representation in $M_n$ and one loses this when passing to quotients.
These are not the kinds of hereditary conditions we want. Rather,
there is another natural kind of ``subobject'' and ``quotient'' besides
ordinary subalgebras and quotient algebras, and requiring them
to be antisymmetric becomes a nontrivial condition (Definition \ref{hadef}).

In infinite dimensions the prospect of a general structure theory is
limited by the possibility that there could
be bounded operators with no nontrivial invariant subspaces. However, assuming
a positive solution to the transitive algebra problem, we can at least show
that any hereditarily antisymmetric dual operator algebra can be put in
upper triangular form, in the infinite dimensional sense of being contained
in the nest algebra for some maximal nest (Theorem \ref{istructure}). The
same technique applied to a single operator shows that if the invariant
subspace problem for Hilbert space operators has a positive solution, then
every bounded operator can be put in upper triangular form, in the same
sense (Theorem \ref{ispthm}).

A word about notation. The operator algebra $M_n$ comes equipped with a
natural involution, namely the Hermitian transpose operation. However, we will
sometimes want to work with matrices relative to some nonorthogonal basis of
$\mathbb{C}^n$, in which case the Hilbert space adjoint of an operator is not
expressed by the Hermitian transpose of its matrix. In these cases I will
write $\widetilde{M}_n$ for the unital algebra of $n\times n$ complex matrices
without any distinguished involution. Thus, results about $\widetilde{M}_n$
will hold for the matrix representations of linear operators on $\mathbb{C}^n$
relative to any, possibly nonorthogonal, basis of $\mathbb{C}^n$.

In a slight abuse of notation, if $A \in B(\mathcal{H})$ and $P$ is the
orthogonal projection onto a closed subspace $\mathcal{E} \subseteq
\mathcal{H}$, I will
often identify $PAP$ with an operator in $B(\mathcal{E})$.

Unless qualified as ``orthogonal'', the word ``projection'' will always mean
``possibly nonorthogonal projection'', i.e., a bounded operator $P$ satisfying
$P = P^2$ but not necessarily $P = P^*$.

Throughout this paper the scalar field is complex. The standard basis of
$\mathbb{C}^n$ will be denoted $(e_i)$.

\section{``Quantum'' preorders}

The idea that unital operator algebras are ``quantum'' preorders arises
from the more general idea that operator spaces in finite dimensions ---
that is, linear subspaces $\mathcal{V}$ of $M_n$ --- are ``quantum'' analogs
of relations on a set with $n$ elements. The usual conditions of reflexivity,
symmetry, and transitivity of a relation correspond to $\mathcal{V}$ being
unital, self-adjoint, and an algebra.
\begin{table}[ht]
\centering
\begin{tabular}{c|c}
$R \subseteq X\times X$&$\mathcal{V} \subseteq M_n$\\
\hline
reflexive&$I \in \mathcal{V}$\\
symmetric&$\mathcal{V} = \mathcal{V}^*$\\
transitive&$\mathcal{V}^2 \subseteq \mathcal{V}$\\[1ex]
\end{tabular}
\caption{Analogy between relations on a finite set $X$ and linear subspaces
of $M_n = M_n(\mathbb{C})$}
\end{table}
This is seen in the fact that the subspace ${\rm span}\{E_{ij}: (i,j) \in R\}$
of $M_n$ induced by a relation $R$ on the set $\{1, \ldots, n\}$ satisfies one
of these algebraic conditions if and only if $R$ satisfies the corresponding
relational condition. ($E_{ij}$ is the matrix with a $1$ in the $(i,j)$ entry
and $0$'s elsewhere.)
Thus unital self-adjoint subalgebras of $M_n$ are regarded as ``quantum
equivalence relations'' (reflexive, symmetric, transitive), operator
systems as ``quantum graphs'' (reflexive, symmetric), and general unital
subalgebras as  ``quantum preordered sets'' (reflexive, transitive). The
idea of operator systems as quantum graphs has been particularly fruitful;
see, e.g., \cite{DSW, W2, W3, W4}.

The adjective ``quantum'' is justified in the latter case by the fact that
the confusability graph in classical error correction becomes a confusability
operator system in quantum error correcion. That is, where classically we
use a graph to catalog which pairs of states might be indistinguishable
after passage through a noisy channel, we would in the setting of quantum
mechanics use an operator system to carry this information. This is explained
in detail in \cite{DSW} and \cite{W4}.

The idea that unital operator algebras are ``quantum'' preorders has a
similar
physical justification. To see this, first consider the classical setting of
a finite state system with phase space $S = \{s_1, \ldots, s_n\}$. Suppose we
have a family of classical channels represented by (left) stochastic matrices
$P^\lambda = (p^\lambda_{ij})$ which can be applied to the system. Here each
$P^\lambda$ represents a probabilistic transformation of $S$, with
$p^\lambda_{ij}$ being the probability of the state $s_j$ transitioning to
the state $s_i$.

The relation ``$s_j$ has a nonzero probability of transitioning to $s_i$
under some $P^\lambda$'' merely describes a directed graph on the vertex set
$S$. But the relation ``there is a sequence of channels
$P^{\lambda_1}, \ldots, P^{\lambda_m}$ under which $s_j$ has
a nonzero probability of transitioning to $s_i$'' is transitive: if there
is some sequence of channels which takes $s_k$ to $s_j$ with nonzero
probability, and another sequence of channels which takes $s_j$ to $s_i$ with
nonzero probability, then the concatenation of the two sequences takes
$s_k$ to $s_i$ with nonzero probability. If we include the identity channel
as the $m = 0$ case then this relation is also reflexive, i.e., it is a
preorder.

This framework could describe an experimental scenario where we have
some family of classical channels which we are able to apply to a finite
state system, and the preorder $s_i \preceq s_j$ reflects our ability to
convert the state $s_j$ into the state $s_i$, with nonzero probability,
by sequentially applying some of the channels which are available to us.
It is not a partial order because it might be possible to transition from
$s_j$ to $s_i$ and then back to $s_j$. However, there may also be
``invariant'' subsets of $S$ with the property that once the state of the
system lies in such a subset it cannot escape. These would correspond to
{\it lower} subsets for the preorder $\preceq$, i.e., subsets $S_0 \subseteq S$
with the property that $j \in S_0$ and $i \preceq j$ implies $i \in S_0$.

Now consider the analogous quantum setup. The pure states of a finite quantum
system are represented as unit vectors in $\mathbb{C}^n$, and the mixed
states by positive unit trace matrices in $M_n$. A quantum channel is a
completely positive trace preserving (CPTP) map $\Phi: M_n \to M_n$,
taking mixed states to mixed states. We can always express $\Phi$ in the
form $\Phi(A) = \sum K_iAK_i^*$ where the {\it Kraus matrices} $K_i$
satisfy $\sum K_i^*K_i = I$.

Given some available set of quantum channels $\Phi^\lambda$, we can ask the
same question as in the classical case: for which unit vectors $v$ and $w$
is there a nonzero probability of transitioning from $v$ to $w$ after the
application of some sequence of $\Phi^\lambda$'s? Where in the classical
setting this information was represented by a preorder on the set
of states, in the quantum setting it will be represented by the unital
algebra $\mathcal{A}$ generated by the Kraus matrices of the available channels
$\Phi^\lambda$. (The Kraus matrices are not uniquely determined by the
$\Phi^\lambda$, but their
linear span is, and hence so is the unital algebra they generate.) Namely,
there is a nonzero probability of transitioning from $v$ to $w$ after
the application of some sequence of $\Phi^\lambda$'s if and only if
we have $\langle Av,w\rangle \neq 0$ for some $A \in \mathcal{A}$. More
generally, if $v$ and $w$ are unit vectors in $\mathbb{C}^n\otimes
\mathbb{C}^k$ for some $k$, representing pure states of some composite
system, then it is possible to transition from $v$ to $w$ if and only if
$\langle (A\otimes I_k)v,w\rangle \neq 0$ for some $A \in \mathcal{A}$. In
fact, this property characterizes $\mathcal{A}$: it is the unique linear
subspace of $M_n$ for which this is true (\cite{W4}, Proposition 6.2).

In the classical setting we also had invariant subsets from which one
could not escape; the analogous quantum notion would be a linear subspace
of $\mathbb{C}^n$ which is invariant for every operator in $\mathcal{A}$.

To summarize: in finite state quantum systems unital operator algebras play
a role analogous to that played by preorders in finite state classical
systems. Thus unital operator algebras are ``quantum'' preorders, in the
same way that operator systems are ``quantum'' graphs \cite{W4}.

\section{Hereditary antisymmetry}

Moving back to the general idea that one can think of linear subspaces of
$M_n$ as being somehow analogous to relations on a set with $n$ elements:
in this picture $\mathbb{C}^n$ plays the role of an $n$ element set on
which a relation is defined, and the linear subspace $\mathcal{V} \subseteq
M_n$ specifies that relation by the condition that two unit vectors $v$ and
$w$ in $\mathbb{C}^n$ are related if $\langle Av,w\rangle \neq 0$ for some
$A \in \mathcal{V}$. From this point of view, the natural notion of a
``subobject'' of $\mathcal{V}$ is its compression to some subspace of
$\mathbb{C}^n$ (cf.\ Section 4 of \cite{W4}). Now if $\mathcal{V}$ is an
algebra, then its compressions $P\mathcal{V}P$, for $P \in M_n$ an orthogonal
projection onto a subspace of $\mathbb{C}^n$, generally are not algebras.
However, in some cases they are.

An {\it invariant subspace} for an operator algebra $\mathcal{A} \subseteq
\mathcal{B}(\mathcal{H})$ is a closed subspace $\mathcal{E}$ of $\mathcal{H}$
with the property that $A(\mathcal{E}) \subseteq \mathcal{E}$ for all
$A \in \mathcal{A}$ (what happens in $\mathcal{E}$ stays in $\mathcal{E}$).
In this paper a {\it co-invariant subspace} will be a subspace whose
orthocomplement is invariant for $\mathcal{A}$, or equivalently a subspace
which is invariant for $\mathcal{A}^*$. Finally, a
{\it semi-invariant subspace} is the orthogonal difference of two invariant
subspaces, i.e., a closed subspace of the form $\mathcal{E}_1 \ominus
\mathcal{E}_2 = \mathcal{E}_1 \cap \mathcal{E}_2^\perp$ where $\mathcal{E}_1$
and $\mathcal{E}_2$ are invariant and $\mathcal{E}_2 \subseteq \mathcal{E}_1$.
Equivalently, $\mathcal{E} \subseteq\mathcal{H}$ is semi-invariant if there
is an orthogonal decomposition $\mathcal{H} = \mathcal{F}_1 \oplus \mathcal{E}
\oplus \mathcal{F}_2$ such that $\mathcal{F}_1 \oplus \mathcal{E}$
is invariant and $\mathcal{E} \oplus \mathcal{F}_2$ is co-invariant.
\begin{figure}[ht]\label{sifig}
$$\bbordermatrix{
&\mathcal{F}_1&\mathcal{E}&\mathcal{F}_2\cr
\mathcal{F}_1&*&*&*\cr
\mathcal{E}&0&*&*\cr
\mathcal{F}_2&0&0&*\cr}
$$
\caption{$\mathcal{F}_1 \oplus \mathcal{E}$ is invariant and
$\mathcal{E} \oplus \mathcal{F}_2$ is coinvariant}
\end{figure}
According to Theorem 2.16 of \cite{D}, a closed subspace $\mathcal{E}$ is
semi-invariant for $\mathcal{A}$ if and only if the map $A \mapsto PAP$ is
a homomorphism from $\mathcal{A}$ into $\mathcal{B}(\mathcal{E})$, i.e.,
$PABP = (PAP)(PBP)$ for all $A,B \in \mathcal{A}$. Here $P$ is the orthogonal
projection onto $\mathcal{E}$. In particular, if $\mathcal{E}$ is
semi-invariant for $\mathcal{A}$ then $P\mathcal{A}P \subseteq
\mathcal{B}(\mathcal{E})$
is still an operator algebra (and it is unital if $\mathcal{A}$ was).

In infinite dimensions we will mainly be interested in weak* closed
operator algebras, necessitating some small modifications in the next
definition. I will defer discussion of this aspect to Section 7. The rest
of the present section deals only with the finite dimensional setting.

\begin{defi}\label{hadef}
Let $\mathcal{A} \subseteq B(\mathbb{C}^n) \cong M_n$ be an operator algebra
and let $P \in M_n$ be the othogonal projection onto a subspace
$\mathcal{E} \subseteq \mathbb{C}^n$. Then $P\mathcal{A}P$ is
\smallskip

{\narrower{
\noindent (i) a {\it subobject} of $\mathcal{A}$ if $\mathcal{E}$ is
invariant for $\mathcal{A}$;

\noindent (ii) a {\it quotient} of $\mathcal{A}$ if $\mathcal{E}$ is
co-invariant for $\mathcal{A}$;

\noindent (iii) a {\it subquotient} of $\mathcal{A}$ if $\mathcal{E}$ is
semi-invariant for $\mathcal{A}$.
\smallskip

}}
\noindent $\mathcal{A}$ is {\it hereditarily antisymmetric} if every
subquotient of $\mathcal{A}$ is antisymmetric.
\end{defi}

There should be no confusion with the term ``quotient'' because in this
paper the word will always be used in the above sense, and never in the
more general sense of ``quotient by an ideal''.

By the comment made just above, subobjects, quotients, and subquotients
are always operator algebras. The intuition for why our subobjects are
rightly thought of as subobjects was explained above. For quotients,
the idea is that if $\mathcal{E}$ is invariant then the action of
$\mathcal{A}$ on $\mathbb{C}^n$ descends to an action on
$\mathbb{C}^n/\mathcal{E} \cong \mathcal{E}^\perp$.

For the definition of hereditary antisymmetry in infinite dimensions
see Definition \ref{ihadef}.

Note that $\mathbb{C}^n$ is a semi-invariant subspace for any
$\mathcal{A} \subseteq M_n$, so that $\mathcal{A} = I\mathcal{A}I$ is
always a subobject and a quotient of itself.

The next proposition is basic.

\begin{prop}\label{compofcomp}
Let $\mathcal{A} \subseteq M_n$ be an operator algebra.
Then any subquotient of a subquotient of $\mathcal{A}$
is a subquotient of $\mathcal{A}$.
\end{prop}

\begin{proof}
Let $P\mathcal{A}P$ be a subquotient of $\mathcal{A}$, where $P$ is
the orthogonal projection onto a semi-invariant subspace $\mathcal{E}$ for
$\mathcal{A}$. Say that $\mathcal{E} = \mathcal{E}_1 \ominus \mathcal{E}_2$
where $\mathcal{E}_1$ and $\mathcal{E}_2$ are invariant subspaces for
$\mathcal{A}$. Then let $QP\mathcal{A}PQ = Q\mathcal{A}Q$ be a subquotient
of $P\mathcal{A}P$, where $Q$ is the orthogonal projection onto
a semi-invariant subspace $\mathcal{F}$ for $P\mathcal{A}P$. Say that
$\mathcal{F} = \mathcal{F}_1 \ominus \mathcal{F}_2$ where $\mathcal{F}_1,
\mathcal{F}_2 \subseteq \mathcal{E}$ are invariant subspaces for
$P\mathcal{A}P$. We must show that $\mathcal{F}$ is semi-invariant for
$\mathcal{A}$; this will imply that $Q\mathcal{A}Q$ is a subquotient
of $\mathcal{A}$.

I claim that $\mathcal{E}_2 + \mathcal{F}_1$ is invariant for
$\mathcal{A}$. To see this, let $v \in \mathcal{E}_2$, $w \in \mathcal{F}_1$,
and $A \in \mathcal{A}$. Then $Aw \in \mathcal{E}_1$ (since $\mathcal{F}_1
\subseteq \mathcal{E}_1$) and $PAw = (PAP)w \in \mathcal{F}_1$ (since
$\mathcal{F}_1$ is invariant for $P\mathcal{A}P$). This shows that $Aw \in
\mathcal{E}_1 \ominus (\mathcal{E}\ominus \mathcal{F}_1) = \mathcal{E}_2 +
\mathcal{F}_1$, and therefore also $A(v+w) = Av + Aw \in \mathcal{E}_2 +
\mathcal{F}_1$. So $\mathcal{E}_2 + \mathcal{F}_1$ is invariant,
as claimed. By the same reasoning, $\mathcal{E}_2 + \mathcal{F}_2$ is
invariant for $\mathcal{A}$, and thus $\mathcal{F} = (\mathcal{E}_2 +
\mathcal{F}_1)\ominus (\mathcal{E}_2 + \mathcal{F}_2)$ is semi-invariant for
$\mathcal{A}$. This is what we needed to show.
\end{proof}

\begin{coro}\label{subha}
Any subquotient of a hereditarily antisymmetric operator algebra
in $M_n$ is hereditarily antisymmetric.
\end{coro}

The following class of examples might give some intuition for the nature of
subobjects, quotients, and subquotients.

\begin{exam}\label{preorderex}
Let $\preceq$ be a preorder on the set $\{1, \ldots, n\}$ and define
$$\mathcal{A}_{\preceq} = {\rm span}\{E_{ij}: i \preceq j\} \subseteq M_n$$
where $E_{ij}$ is the matrix with a $1$ in the $(i,j)$ entry and $0$'s
elsewhere.  This is an algebra because $E_{ij}E_{jk} = E_{ik}$ and
$i \preceq j$, $j \preceq k$ $\Rightarrow$ $i \preceq k$. It is unital
because $I = E_{11} + \cdots + E_{nn}$.

Suppose $\mathcal{E} \subseteq \mathbb{C}^n$ is an invariant subspace for
$\mathcal{A}_{\preceq}$. For any $v \in \mathcal{E}$ and $1 \leq i \leq n$
we must have $E_{ii}v \in \mathcal{E}$, and this shows that $\mathcal{E}$
must be the span of some subset of the standard basis $\{e_1, \ldots, e_n\}$.
Thus, invariant subspaces for $\mathcal{A}_{\preceq}$ correspond to certain
subsets of $\{1, \ldots, n\}$. A moment's thought shows that the subsets of
$\{1, \ldots, n\}$ which correspond to invariant subspaces are precisely the
lower sets, i.e., sets $X$ which satisfy $i \preceq j \in X$
$\Rightarrow$ $i \in X$, while the subsets which correspond to orthocomplements
of invariant subspaces are precisely the upper sets satisfying the opposite
condition. The semi-invariant subspaces for
$\mathcal{A}_{\preceq}$ are therefore the subspaces of the form
${\rm span}\{e_i: i \in X\}$ where $X \subseteq \{1, \ldots, n\}$ is the
difference of two lower sets. This is equivalent to saying that $X$ is
convex, i.e., $i \preceq j \preceq k$ and $i,k \in X$ $\Rightarrow$
$j \in X$.

Thus, the subobjects of $\mathcal{A}$ correspond to lower subsets of
$\{1, \ldots, n\}$ under $\preceq$, the quotients correspond to upper
subsets, and the subquotients correspond to convex subsets.
\end{exam}

In this example the ordinary subalgebras of $\mathcal{A}_{\preceq}$ which
contain all the diagonal matrices correspond to preorders on
$\{1, \ldots, n\}$ which are weaker than $\preceq$ (cf.\ Theorem \ref{precthm}
below).

Sometimes a modified version of a semi-invariant subspace, which is not
orthogonal to $\mathcal{E}_2$ (in the notation from the beginning of this
section), is more natural. In those situations the following notion can be
helpful.

\begin{defi}\label{compdef}
Let $\mathcal{A} \subseteq M_n$ be an operator algebra and let
$\mathcal{E} = \mathcal{E}_1 \ominus \mathcal{E}_2$ be a semi-invariant
subspace for $\mathcal{A}$, where $\mathcal{E}_1$ and $\mathcal{E}_2$ are
invariant subspaces. A {\it companion subspace} of $\mathcal{E}$ (relative
to the expression of $\mathcal{E}$ as $\mathcal{E}_1 \ominus \mathcal{E}_2$,
but I will take this as understood) is any
complementary subspace $\mathcal{F}$ of $\mathcal{E}_2$ in $\mathcal{E}_1$.
That is, $\mathcal{F}$ is any linear subspace of $\mathcal{E}_1$ satisfying
$\mathcal{E}_2 + \mathcal{F} = \mathcal{E}_1$ and $\mathcal{E}_2 \cap
\mathcal{F} = \{0\}$. The {\it natural projection} onto $\mathcal{F}$ is
the (nonorthogonal) projection onto $\mathcal{F}$ whose kernel is
$\mathcal{E}_2 \oplus \mathcal{E}_1^\perp$.
\end{defi}

\begin{prop}\label{compprop}
Let $\mathcal{E} = \mathcal{E}_1\ominus\mathcal{E}_2$ be a semi-invariant
subspace for an operator algebra $\mathcal{A} \subseteq M_n$ and
let $\mathcal{F}$ be a companion subspace of $\mathcal{E}$. Let $P$ be the
orthogonal projection onto $\mathcal{E}$, let $P_0: \mathcal{F} \to
\mathcal{E}$ be its restriction to $\mathcal{F}$, and let
$Q \in M_n$ be the natural projection onto $\mathcal{F}$. Then
$\Phi: T \mapsto P_0TP_0^{-1}$ defines an isomorphism between $Q\mathcal{A}Q
\subseteq B(\mathcal{F})$ and $P\mathcal{A}P \subseteq B(\mathcal{E})$.
\end{prop}

\begin{proof}
I claim that $\Phi(QAQ) = PAP$ for all $A \in \mathcal{A}$; this will show that
$\Phi$ maps $Q\mathcal{A}Q$ bijectively onto $P\mathcal{A}P$. The fact that
$\Phi$ is an algebra homomorphism is straightforward.

To prove the claim, observe first that for any $v \in \mathcal{E}_1$ we have
$Pv, Qv \in v + \mathcal{E}_2$ since ${\rm ker}(P) \cap \mathcal{E}_1
= {\rm ker}(Q) \cap \mathcal{E}_1 = \mathcal{E}_2$. Likewise for $P_0v$
when $v \in \mathcal{F}$ and $P_0^{-1}v$ when $v \in \mathcal{E}$. So fixing
$A \in \mathcal{A}$ and $v \in \mathcal{E}$, we have
$QP_0^{-1}v = P_0^{-1}v \in v + \mathcal{E}_2$, and then $AQP_0^{-1}v
\in Av + \mathcal{E}_2$, and then $QAQP_0^{-1}v \in Av + \mathcal{E}_2$ and
$P_0QAQP_0^{-1}v \in Av + \mathcal{E}_2$. But also
$PAPv = PAv \in Av + \mathcal{E}_2$, so that
$$PAPv - \Phi(QAQ)v \in \mathcal{E}_2 \cap \mathcal{E} = \{0\}.$$
Since $v$ was arbitrary, this shows that $\Phi(QAQ) = PAP$.
\end{proof}

\section{Basic facts}

Before we discuss hereditarily antisymmetric operator algebras, it will be
helpful to collect some basic facts about operator algebras, both
antisymmetric and not. These results mostly pertain to the finite dimensional
setting which will be our primary focus.

The first result, however, applies to both the finite and infinite
dimensional settings. It provides a simple alternative characterization
of antisymmetry. This was Proposition 1 (i, iv) of \cite{Sz}, but I include
the proof for the sake of completeness.

\begin{prop}\label{sainf}
Let $\mathcal{A} \subseteq \mathcal{B}(\mathcal{H})$ be an operator algebra.
Then $\mathcal{A}$ is antisymmetric if and only if every self-adjoint element
of $\mathcal{A}$ is a scalar multiple of the identity. If $\mathcal{A}$ is
weak* closed, then it is antisymmetric if and only if it contains no
orthogonal projections besides $0$ and (possibly) $I$.
\end{prop}

\begin{proof}
The forward implication in the first assertion is trivial. For the reverse
implication, suppose
$\mathcal{A}$ is not antisymmetric. Then there is some $A \in \mathcal{A}
\cap \mathcal{A}^*$ which is not a scalar multiple of the identity. Since
both $A$ and its adjoint belong to $\mathcal{A}$, its real and imaginary
parts ${\rm Re}(A) = \frac{1}{2}(A + A^*)$ and ${\rm Im}(A) =
\frac{1}{2i}(A - A^*)$ both also belong to $\mathcal{A}$, and (since $A =
{\rm Re}(A) + i\cdot {\rm Im}(A)$) they cannot both be scalar multiplies of
the identity. So $\mathcal{A}$ contains a self-adjoint operator which is not
a scalar multiple of the identity.

The forward implication in the second assertion is also trivial. For the
reverse implication, suppose $\mathcal{A}$ is not antisymmetric. Then it
contains a non-scalar self-adjoint operator $A$ by the first part of the
proposition. Since $\mathcal{A}$ is weak* closed, it then contains the von
Neumann algebra generated by $A$, and hence it must contain a nontrivial
orthogonal projection.
\end{proof}

Next, we show that antisymmetric algebras can always be unitized. This works
in infinite dimensions, too.

\begin{prop}\label{unitprop}
Let $\mathcal{A} \subset \mathcal{B}(\mathcal{H})$ be a nonunital operator
algebra. If $\mathcal{A}$ is antisymmetric then so is its unitization
$\mathcal{A} + \mathbb{C}\cdot I$.
\end{prop}

\begin{proof}
Suppose $\mathcal{A} + \mathbb{C}\cdot I$ is not antisymmetric. Then
according to Proposition \ref{sainf} there exists a self-adjoint operator
of the form $A + aI$ with $A \in \mathcal{A}$ nonzero. Then $A + aI =
A^* + \bar{a}I$ implies that $A = A^* + bI$ where $b = \bar{a} - a$. But then
$A^2 = A^*A + bA \in \mathcal{A}$, and hence $A^*A \in \mathcal{A}$. Thus
$\mathcal{A}$ contains a nonzero self-adjoint operator, which means that it
cannot be antisymmetric either.
\end{proof}

Proposition \ref{sainf} can be strengthened in the finite dimensional setting.
We need the following lemma, which is probably well-known.

\begin{lemma}\label{linflemma}
Every subalgebra of $l^\infty_n = l^\infty(\{1, \ldots, n\})$ is self-adjoint.
\end{lemma}

\begin{proof}
Let $\mathcal{A}$ be a subalgebra of $l^\infty_n$, let $f \in \mathcal{A}$,
and let $X = {\rm Ran}(f) \cup \{0\}$. For each nonzero $a \in X$, find a
complex polynomial $p$ satisfying $p(a) = 1$ and $p(b) = 0$ for all other
$b \in X$. Then $p$ has no constant term, so $f_a = p\circ f \in \mathcal{A}$.
This shows that we can write $f = \sum_{a \in X\setminus\{0\}} af_a$ where
each $f_a$ belongs to $\mathcal{A}$ and takes only the values $0$ and $1$.
Thus $\bar{f} = \sum \bar{a}f_a$ also belongs to $\mathcal{A}$. We conclude
that $\mathcal{A}$ is self-adjoint.
\end{proof}

The situation for $l^\infty = l^\infty(\mathbb{N})$ could not be more
different; see Proposition \ref{hinfty}.

It is standard that the self-adjoint unital subalgebras of $l^\infty_n$
correspond to the equivalence relations on $\{1, \ldots, n\}$, by associating
an equivalence relation $\sim$ to the set of functions in $l^\infty_n$ which
are constant on each block of $\sim$. Any nonunital self-adjoint subalgebra
is, for some equivalence relation $\sim$, the set of all functions which are
constant on each block of $\sim$ and which vanish on some specified block.

\begin{prop}\label{tfaefd}
Let $\mathcal{A} \subseteq M_n$ be an operator algebra. The following are
equivalent:

{\narrower{
\noindent (i) $\mathcal{A}$ is antisymmetric

\noindent (ii) every self-adjoint element of $\mathcal{A}$ is a scalar
multiple of the identity

\noindent (iii) there are no orthogonal projections in $\mathcal{A}$
besides $0$ and (possibly) $I$

\noindent (iv) every unitary element of $\mathcal{A} + \mathbb{C}\cdot I$
is a scalar multiple of the identity

\noindent (v) every normal element of $\mathcal{A}$ is a scalar multiple of
the identity.

}}
\end{prop}

\begin{proof}
The equivalence of (i), (ii), and (iii) was shown in Proposition \ref{sainf}.

For (i) $\Rightarrow$ (v), suppose $\mathcal{A}$ contains a normal
operator $A$ which is not a scalar multiple of the identity. Working in an
orthonormal basis which diagonalizes $A$, the operator algebra generated by
$A$ constitutes
a subalgebra of the diagonal subalgebra of $M_n$, which can be identified with
$l^\infty_n$. We infer from Lemma \ref{linflemma} that $A^*$ also belongs to
this subalgebra, and hence that $A^* \in \mathcal{A}$. So $A \in \mathcal{A}
\cap \mathcal{A}^*$, showing that $\mathcal{A}$ is not antisymmetric.

For (v) $\Rightarrow$ (iv), suppose there is a non-scalar unitary
$U = A + \alpha I$ with $A \in \mathcal{A}$. Then $A$ is a non-scalar
normal operator in $\mathcal{A}$.

For (iv) $\Rightarrow$ (ii), suppose there is a self-adjoint operator
$A$ in $\mathcal{A}$ which is not a scalar multiple of the identity. Then
$e^{iAt} \in \mathcal{A} + \mathbb{C}\cdot I$ for every $t \in \mathbb{R}$;
one can infer this from Lemma \ref{linflemma} or simply consider the power
series expansion of $e^{iAt}$. This is a one-parameter unitary group, and
since $A = \lim_{t \to 0}\frac{1}{it}(e^{iAt} - I)$, the operators $e^{iAt}$
cannot all be scalar multiples of the identity. (Alternatively, one can
deduce the implication (iv) $\Rightarrow$ (ii) by applying the comment made
after Lemma \ref{linflemma} to the algebra generated by $A$ in an orthonormal
basis which diagonalizes $A$.)
\end{proof}

The equivalence with conditions (iv) and (v) fails in infinite
dimensions; see Example \ref{inormal}.

Next we note a simple dimensional restriction on antisymmetric subalgebras
of $M_n$.

\begin{prop}\label{dimcount}
Suppose an operator algebra $\mathcal{A} \subseteq M_n$ has dimension at
least $\frac{n^2}{2} + 1$. Then $\mathcal{A}$ is not antisymmetric.
\end{prop}

\begin{proof}
$M_n$ has real dimension $2n^2$, and the set of self-adjoint $n\times n$
matrices with zero trace is a real linear subspace of $M_n$ of real dimension
$n^2 - 1$. Then any complex linear subspace of $M_n$ whose complex dimension
is at least $\frac{n^2}{2} + 1$ will have real dimension at least $n^2 + 2$
and hence must nontrivially intersect the space of self-adjoint matrices with
zero trace. Thus no operator algebra of dimension at least $\frac{n^2}{2} + 1$
can be antisymmetric.
\end{proof}

Putting operator algebras in upper triangular form, i.e., finding a basis
with respect to which the matrix of every element of the algebra is upper
triangular, will be a recurring theme in this paper. The next result is basic.
In order for the notion of ``upper triangular'' to be meaningful, we need not
merely a basis, but an {\it ordered} basis in which the basis vectors appear
in a specified order.

\begin{prop}\label{orthout}
Let $\mathcal{A} \subseteq M_n$ be an operator algebra and suppose there
is a (possibly nonorthogonal) ordered basis of $\mathbb{C}^n$ with respect
to which the matrix of every element of $\mathcal{A}$ is upper triangular.
Then there is an ordered
orthonormal basis with the same property. If in addition the matrix
with respect to the first basis of every element of $\mathcal{A}$ is constant
on its main diagonal, then the same will be true of the matrix with respect
to the second basis of every element of $\mathcal{A}$.
\end{prop}

\begin{proof}
Suppose $(v_1, \ldots, v_n)$ is an ordered basis with respect to which the
matrix of every element of $\mathcal{A}$ is upper triangular. This means that
$Av_i \in {\rm span}\{v_1, \ldots, v_i\}$ for all $A \in \mathcal{A}$
and $1 \leq i \leq n$, or, to put it differently, for every $i$ the subspace
${\rm span}\{v_1, \ldots, v_i\}$ is invariant for $\mathcal{A}$. Applying
the Gram-Schmidt orthonormalization procedure to this basis produces an
ordered orthonormal basis $(w_1, \ldots, w_n)$ with the property that
${\rm span}\{w_1, \ldots, w_i\} = {\rm span}\{v_1, \ldots, v_i\}$ for all
$1 \leq i \leq n$. So each ${\rm span}\{w_1, \ldots, w_i\}$ is invariant
for $\mathcal{A}$, and this shows that every element of $\mathcal{A}$
has an upper triangular matrix for the $(w_i)$ basis.

The second assertion follows from the fact that the main diagonal entries
of an upper triangular matrix are precisely its eigenvalues. So if the matrix
of $A \in \mathcal{A}$ with respect to the $(v_i)$ basis is upper
triangular and constant on its main diagonal, then $A$ has only one eigenvalue
and so its matrix with respect to the $(w_i)$ basis will also be
constant on its main diagonal.
\end{proof}

The final theorem of this section shows that in finite dimensions, operator
algebras typically contain plenty of nonorthogonal projections (with the
exceptions described below in Theorem \ref{tfaeex1}). As stated,
it applies to block diagonal matrices in which each block is upper triangular
and constant on its main diagonal --- the type just treated in Proposition
\ref{orthout}. Pictorially, these are matrices of the form shown in Figure
\ref{jordanesque}.
\begin{figure}[ht]
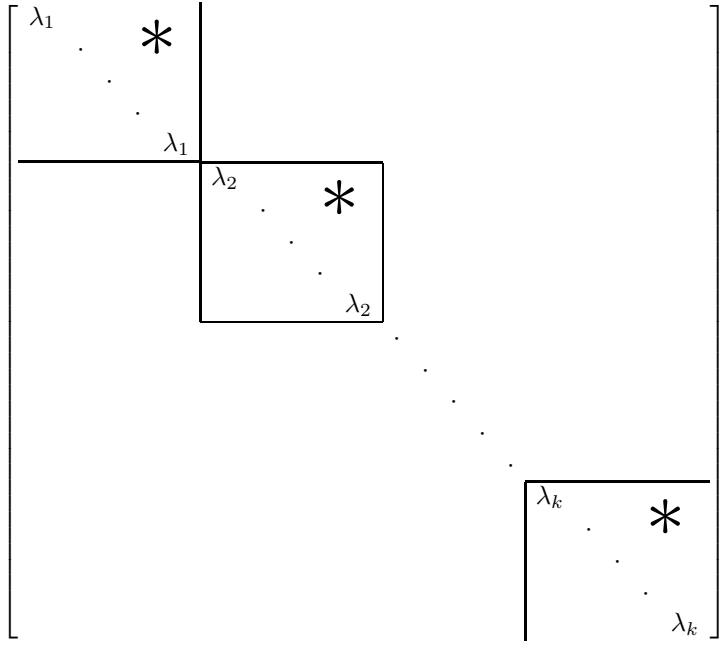
\label{jordanesque}
$$\setlength{\arraycolsep}{0pt}
\left[\begin{matrix}
\begingroup
\setlength{\arraycolsep}{4pt}
\begin{array}{ccccc|}
\lambda_1&&&&\\
&\cdot&&\hspace*{.1in}\smash{\makebox[\wd0]{\HUGE $*$}}\hspace{-.1in}&\\
&&\cdot&&\\
&&&\cdot&\\
&&&&\lambda_1\\
\hline
\end{array}\endgroup &&&\\
&\begingroup
\setlength{\arraycolsep}{4pt}
\begin{array}{|ccccc|}
\hline
\lambda_2&&&&\\
&\cdot&&\hspace*{.1in}\smash{\makebox[\wd0]{\HUGE $*$}}\hspace{-.1in}&\\
&&\cdot&&\\
&&&\cdot&\\
&&&&\lambda_2\\
\hline
\end{array}\endgroup &&\\
&&\begingroup
\setlength{\arraycolsep}{4pt}
\begin{array}{ccccc}
\cdot&&&&\\
&\cdot&&&\\
&&\cdot&&\\
&&&\cdot&\\
&&&&\cdot\\
\end{array}\endgroup &\\
&&&\begingroup
\setlength{\arraycolsep}{4pt}
\begin{array}{|ccccc}
\hline
\lambda_k&&&&\\
&\cdot&&\hspace*{.1in}\smash{\makebox[\wd0]{\HUGE $*$}}\hspace{-.1in}&\\
&&\cdot&&\\
&&&\cdot&\\
&&&&\lambda_k\\
\end{array}\endgroup \\
\end{matrix}\right]$$
\caption{A Jordanesque matrix}
\end{figure}
Jordan matrices are particular examples of matrices of this type.
Therefore I will call any matrix of the above form ``Jordanesque''. Let
us record this in a definition.

\begin{defi}\label{jordef}
A block diagonal matrix in $\widetilde{M}_n$ in which each block is upper
triangular and constant on its main diagonal is {\it Jordanesque}. If the
sizes of the blocks are $n_1, \ldots, n_k$ (so that $n_1 + \cdots + n_k = n$)
then we may also say the matrix is {\it $(n_1, \ldots, n_k)$-Jordanesque}.
\end{defi}

Any linear operator on $\mathbb{C}^n$ has a matrix in Jordan form relative
to some ordered basis of $\mathbb{C}^n$. But this basis need not be orthogonal,
so it is important that the preceding definition applies to
$\widetilde{M}_n$, which has no preferred involution.

It is straightforward to check that the sum and product of any two 
$(n_1, \ldots, n_k)$-Jordanesque matrices are again
$(n_1, \ldots, n_k)$-Jordanesque. Thus the set of all
$(n_1, \ldots, n_k)$-Jordanesque matrices is a unital operator algebra.

The crux of the next proof is the simple fact that if $A \in \widetilde{M}_n$
is strictly upper triangular (i.e., upper triangular and zero on the main
diagonal) then $A^n = 0$.

\begin{theo}\label{findproj}
Let $A \in \widetilde{M}_n$ be Jordanesque and let $\lambda$ be one of its
diagonal entries (i.e., one of its eigenvalues). Assume $\lambda \neq 0$.
Then the algebra generated by $A$ contains the diagonal matrix whose $(i,i)$
entry is $1$ if $a_{ii} = \lambda$ and is $0$ otherwise, where $A = (a_{ij})$.
\end{theo}

\begin{proof}
Let $\mathcal{A}$ be the algebra generated by $A$. For any nonzero eigenvalue
$\mu$ besides $\lambda$, the matrix $A^2 - \mu A \in \mathcal{A}$ has zero
diagonal on every block where $A$ has diagonal entries $\mu$, but its
diagonal entries on every block where $A$ has diagonal entries $\lambda$
are nonzero. The restriction of $A^2 - \mu A$ to any of the former blocks
is strictly upper triangular, so that $(A^2 - \mu A)^n$ is zero on all of these
blocks. Thus, if $\mu_1, \ldots, \mu_k$ are the nonzero eigenvalues of $A$
besides $\lambda$, then $(A^2 - \mu_1 A)^n \cdots (A^2 - \mu_k A)^n \in
\mathcal{A}$ is nonzero only on blocks where $A$ has diagonal entries
$\lambda$, and on those blocks its diagonal entries all take the nonzero
value $\lambda ' = \lambda^{kn}(\lambda - \mu_1)^n\cdots(\lambda - \mu_k)^n$.

Multiplying by $\frac{1}{\lambda'}$ and restricting to the nonzero blocks,
we reduce to the case where $A$ is upper triangular and its diagonal entries
are all $1$. We must now show that $I \in \mathcal{A}$. Write $A = I + A_0$
where $A_0$ is strictly upper triangular. For any $B \in \mathcal{A}$ we
have $A_0B = AB - B \in \mathcal{A}$. So inductively $A_0 + A_0^2$,
$A_0^2 + A_0^3$, $\ldots$, $A_0^{n-1} + A_0^n = A_0^{n-1}$ all belong to
$\mathcal{A}$. Then the alternating sum
$$(I + A_0) - (A_0 + A_0^2) + \cdots \pm (A_0^{n-2} + A_0^{n-1})
\mp A_0^{n-1} = I$$
belongs to $\mathcal{A}$, as desired.
\end{proof}

\begin{coro}\label{bigdiag}
If $\mathcal{A} \subseteq M_n$ is an operator algebra all of whose elements
are Jordanesque, then we have $\mathcal{A} = \mathcal{A}_{\rm diag} +
\mathcal{A}_{\rm nil}$ where $\mathcal{A}_{\rm diag}$ is the algebra of
diagonal matrices in $\mathcal{A}$ and $\mathcal{A}_{\rm nil}$ is the ideal
of strictly upper triangular (i.e., nilpotent) matrices in $\mathcal{A}$.
\end{coro}

\begin{proof}
Let $A \in \mathcal{A}$. Then applying Theorem \ref{findproj} to each nonzero
eigenvalue of $A$ and taking a linear combination, we get that the diagonal
matrix $B$ whose diagonal entries agree with those of $A$ belongs
to $\mathcal{A}$. Thus every matrix in $\mathcal{A}$ can be decomposed into
the sum of a diagonal matrix and a strictly upper triangular matrix, both
of which belong to $\mathcal{A}$. This shows that $\mathcal{A} =
\mathcal{A}_{\rm diag} + \mathcal{A}_{\rm nil}$.
\end{proof}

Since any linear operator on $\mathbb{C}^n$ can be put in Jordan form by a
suitable choice of (not necessarily orthogonal) basis, the preceding results
apply to any linear operator on $\mathbb{C}^n$ and the algebra it generates.

\section{Two basic examples}

There are two prototypical examples of hereditarily antisymmetric operator
algebras in $M_n$. We will see in Theorem \ref{mainthm} that any hereditarily
antisymmetric operator algebra in $M_n$ is a sort of combination of these two
types.

\begin{exam}\label{ex1}
Let $n \in \mathbb{N}$ and let $\mathcal{T}_n \subseteq M_n$ be the set of
upper triangular matrices which are constant on the main diagonal.
\begin{figure}[ht]\label{ex1fig}
$$\left[\begin{matrix}
\lambda&&&&\\
&\cdot&&\hspace*{.1in}\smash{\makebox[\wd0]{\HUGE $*$}}\hspace{-.1in}&\\
&&\cdot&&\\
&&&\cdot&\\
\hspace*{.1in}\smash{\makebox[\wd0]{\Huge $0$}}\hspace{-.1in}&&&&\lambda\\
\end{matrix}\right]$$
\caption{Upper triangular and constant on the main diagonal}
\end{figure}
This is a unital operator algebra, and it is clearly antisymmetric.
\end{exam}

In this example we are working in the standard basis of $\mathbb{C}^n$, which
is orthonormal. (That is, $\mathcal{T}_n$ is defined in terms of $M_n$, not
$\widetilde{M}_n$.) But the same definition could be made relative to any
ordered vector space basis of $\mathbb{C}^n$: let
$(v_1, \ldots, v_n)$ be any ordered basis
of $\mathbb{C}^n$ and let $\mathcal{A}$ be the set of operators whose matrices
relative to this basis are upper triangular and constant on the main diagonal.
However, as a consequence of Proposition \ref{orthout} this construction is no
more general than the orthonormal case. That result implies that this new
algebra $\mathcal{A}$ will be identical to the algebra $\mathcal{T}_n$ with
respect to some ordered orthonormal basis. So we have one example of this type
for each dimension $n$, and allowing nonorthonormal bases does not expand the
class of examples.

\begin{prop}\label{ex1maxha}
For each $n \in \mathbb{N}$ the operator algebra $\mathcal{T}_n$ of Example
\ref{ex1} is hereditarily antisymmetric, and every operator algebra that
properly contains it is not antisymmetric. In particular, it is a maximal
hereditarily antisymmetric operator algebra.
\end{prop}

\begin{proof}
In order to check the first assertion, we must identify the semi-invariant
subspaces for $\mathcal{T}_n$. But $\mathcal{T}_n$ contains the shift matrix
\begin{figure}[ht]\label{shiftfig}
$$U = \left[\begin{matrix}
0&1&&\raisebox{-.11in}{\hspace*{.1in}\smash{\makebox[\wd0]{\Huge $0$}}\hspace{-.1in}}&\\
&0&1&&\\
&&\ddots&\ddots&\\
&&&0&1\\
\hspace*{.1in}\smash{\makebox[\wd0]{\Huge $0$}}\hspace{-.1in}&&&&0
\end{matrix}\right]$$
\caption{The shift matrix}
\end{figure}

\noindent whose only invariant subspaces are $\{0\}$, ${\rm span}\{e_1\}$,
${\rm span}\{e_1, e_2\}$, $\ldots$, $\mathbb{C}^n$. So these are the only
possible invariant subspaces for $\mathcal{T}_n$. Conversely, it is easy to
see that each of these subspaces is invariant for any operator in
$\mathcal{T}_n$. So these are precisely the invariant subspaces for
$\mathcal{T}_n$, and the nonzero semi-invariant
subspaces are therefore those of the form ${\rm span}\{e_i, \ldots, e_j\}$
for $i \leq j$. The compression of $\mathcal{T}_n$ to any such subspace is
simply a lower dimensional version of $\mathcal{T}_n$, and hence is
antisymmetric. So $\mathcal{T}_n$ is hereditarily antisymmetric.

Next, we show that $\mathcal{T}_n$ is maximal antisymmetric. Suppose
$\mathcal{A} \subseteq M_n$ is an operator algebra that properly
contains $\mathcal{T}_n$ and let $A \in \mathcal{A} \setminus \mathcal{T}_n$.
There are two cases to consider. First, suppose $A$ is upper triangular
but not constant on the main diagonal. By subtracting an element of
$\mathcal{T}_n$, we may assume that $A$ has no entries above the main
diagonal, i.e., it is a diagonal matrix. But then it is normal but not
a scalar multiple of the identity, which shows that $\mathcal{A}$ is not
antisymmetric by Proposition \ref{tfaefd}.

In the second case $A$ has some nonzero entry $a_{ij}$ with $i > j$. Choose
such a pair $(i,j)$ with $i - j$ maximal. Then $AU^{i-j}$ is upper triangular,
its $(0,0)$ entry is $0$, and its $(i,i)$ entry is $a_{ij}$. (The matrix
$U$ was introduced in the first part of this proof.) So we reduce to
the first case. This completes the proof that $\mathcal{T}_n$ is maximal
antisymmetric.

Since there is no larger antisymmetric algebra in $M_n$, there is certainly
no larger hereditarily antisymmetric algebra in $M_n$. Thus $\mathcal{T}_n$ is
also maximal hereditarily antisymmetric.
\end{proof}

The operator algebras $\mathcal{T}_n$ are quite special. They can be
characterized in several alternative ways. Recall that Burnside's
theorem on matrix algebras states that any operator algebra properly
contained in $M_n$ has a nontrivial invariant subspace.

\begin{theo}\label{tfaeex1}
Let $\mathcal{A} \subseteq M_n$ be an operator algebra.
The following are equivalent:

{\narrower{
\noindent (i) $\mathcal{A}$ is contained in the algebra $\mathcal{T}_n$
relative to some ordered orthonormal basis

\noindent (ii) every similar operator algebra (i.e., every
$S\mathcal{A}S^{-1}$ for $S \in M_n$ invertible) is antisymmetric

\noindent (iii) $\mathcal{A}$ is antisymmetric with repect to any inner
product on $\mathbb{C}^n$

\noindent (iv) $\mathcal{A}$ contains no projections besides $0$ and
(possibly) $I$

\noindent (v) every operator in $\mathcal{A}$ has exactly one eigenvalue.

}}
\end{theo}

\begin{proof}
(i) $\Rightarrow$ (ii): Let $S \in M_n$ be invertible; we must show that
$S\mathcal{T}_nS^{-1}$ is antisymmetric. Now $S\mathcal{T}_nS^{-1}$
consists of the operators whose matrices relative to the $(Se_i)$ basis
are upper triangular and constant on the main diagonal. According to
Proposition \ref{orthout} this means that $S\mathcal{T}_nS^{-1}$ is just
another $\mathcal{T}_n$ with respect to some other orthonormal basis.
This shows that there are no non-scalar
self-adjoint operators in $S\mathcal{T}_nS^{-1}$.

(ii) $\Rightarrow$ (iii): Assume every operator algebra similar to
$\mathcal{A}$ is antisymmetric and let $\{\cdot,\cdot\}$ be a new inner
product on $\mathbb{C}^n$. Then there is an invertible positive operator $T$
with the property that $\langle Tv,w\rangle = \{v,w\}$ for all
$v,w \in \mathbb{C}^n$. Let $S = T^{-1/2}$. Denoting the adjoint operation
relative to $\{\cdot,\cdot\}$ by $A^\dag$, we then have, for any $v$ and $w$,
$$\langle TA^\dag v, w\rangle = \{A^\dag v, w\} = \{v,Aw\}
= \langle Tv,Aw\rangle = \langle A^*Tv,w\rangle.$$
Thus
$$A^\dag = T^{-1}A^*T = S^2A^*S^{-2} = S(S^{-1}AS)^*S^{-1},$$
and hence $S^{-1}A^\dag S = (S^{-1}AS)^*$. So $A^\dag = A \in \mathcal{A}$
implies
$S^{-1}AS = (S^{-1}AS)^*$, which by hypothesis implies that $S^{-1}AS$ is a
scalar multiple of the identity, which implies that $A$ is a scalar multiple
of the identity. Thus by Proposition \ref{sainf} $\mathcal{A}$ is
antisymmetric relative to $\{\cdot,\cdot\}$.

(iii) $\Rightarrow$ (iv): Suppose $\mathcal{A}$ contains a projection $P$
which is neither $0$ nor $I$. Find bases of ${\rm ker}(P)$ and ${\rm ran}(P)$;
together these form a basis for $\mathbb{C}^n$ which, if taken to be
orthonormal, makes $P$ an orthogonal projection. So relative to the
inner product which makes the chosen basis orthonormal, $\mathcal{A}$
contains a non-scalar self-adjoint operator, i.e., it is not antisymmetric.

(iv) $\Rightarrow$ (v): If $\mathcal{A}$ contains an operator with more than
one eigenvalue, then it contains a non-scalar projection by Theorem
\ref{findproj}.

(v) $\Rightarrow$ (i): Suppose every operator in $\mathcal{A}$ has exactly
one eigenvalue. We prove that there is an ordered orthonormal basis relative
to which every operator in $\mathcal{A}$ is upper triangular. The fact that
every operator in $\mathcal{A}$ has exactly one eigenvalue then implies that
these upper triangular matrices are all constant on the main diagonal.

The proof goes by induction on $n$, where $\mathcal{A} \subseteq M_n$. The
assertion is trivial for $n = 1$.
For $n > 1$, let $\mathcal{A}$ be such an algebra and observe that every
subquotient of $\mathcal{A}$ has the same property that every
operator in it has only one eigenvalue. (If $\mathcal{H} =
\mathcal{F}_1 \oplus \mathcal{E} \oplus \mathcal{F}_2$ is an orthogonal
decomposition such that $\mathcal{F}_1$ and $\mathcal{F}_1 \oplus \mathcal{E}$
are invariant for $A$, then the eigenvalues of $A$ are the eigenvalues of its
compressions to $\mathcal{E}$, $\mathcal{F}_1$, and $\mathcal{F}_2$; see
Figure \ref{sifig}.) By Burnside's theorem, $\mathcal{A}$ has a nontrivial
invariant subspace $\mathcal{E}$. Then the induction hypothesis applies to
the compressions of $\mathcal{A}$ to $\mathcal{E}$ and $\mathcal{E}^\perp$,
so we can find ordered
orthonormal bases of $\mathcal{E}$ and $\mathcal{E}^\perp$
such that the matrix of the compression of any operator in $\mathcal{A}$
to either $\mathcal{E}$ or $\mathcal{E}^\perp$ is upper triangular. Then the
matrix of any operator in $\mathcal{A}$ has the same property relative to the
concatenation of these two bases. This proves (i).
\end{proof}

The proof of (v) $\Rightarrow$ (i) appears in a more general form in Theorem
\ref{utthm} below.

Since any subquotient of $\mathcal{A}$ would inherit property (v), the
following corollary is immediate.

\begin{coro}\label{tfaecor}
Every subalgebra of $\mathcal{T}_n$ is hereditarily antisymmetric.
\end{coro}

Now let us turn to the second prototypical class of examples.

\begin{defi}\label{exam2}
Let ${\bf v} = \{v_1, \ldots, v_n\}$ be a basis for $\mathbb{C}^n$ and
define $\mathcal{D}_{\bf v} \subseteq M_n$ to be the set of operators
for which each $v_i$ is an eigenvalue. Equivalently, these are the
operators whose matrix for the ${\bf v}$ basis is diagonal.
\end{defi}

If the ${\bf v}$ basis is orthogonal then $\mathcal{D}_{\bf v}$ clearly
contains many self-adjoint operators. In order to make $\mathcal{D}_{\bf v}$
antisymmetric we need some amount of nonorthogonality. Define the
{\it nonorthogonality graph associated to ${\bf v}$} to have vertices
$\{1, \ldots, n\}$ and an edge between $i$ and $j$ if $\langle v_i,v_j\rangle
\neq 0$. Say that ${\bf v}$ is {\it anti-orthogonal} if, for any
$1 \leq i < j \leq n$ and
$X \subseteq \{1, \ldots, \hat{i}, \ldots, \hat{j}, \ldots, n\}$
(i.e., $\{1, \ldots, n\}$ with $i$ and $j$ omitted)
we have $\langle Pv_i,v_j\rangle \neq 0$ where $P$ is the orthogonal projection
onto the orthocomplement of ${\rm span}\{v_k: k \in X\}$. (In particular,
taking $X = \emptyset$ yields $\langle v_i,v_j\rangle \neq 0$ for all
$i$ and $j$, i.e., the nonorthogonality graph is complete.)

\begin{theo}\label{ex2thm}
Let ${\bf v}$ be a basis for $\mathbb{C}^n$, $n > 1$. Then
$\mathcal{D}_{\bf v}$ is antisymmetric if and only if the nonorthogonality
graph associated to ${\bf v}$ is connected. It is hereditarily antisymmetric
if and only if ${\bf v}$ is anti-orthogonal.
\end{theo}

\begin{proof}
Suppose the nonorthogonality graph associated to ${\bf v}$ is disconnected.
Then there is a nonempty proper subset $X \subset \{1, \ldots, n\}$ such
that $\langle v_i,v_j\rangle = 0$ for any $i \in X$ and $j \not\in X$. The
operator $A$ which satisfies $Av_i = v_i$ for all $i \in X$ and $Av_j = 0$
for all $j \not\in X$ is then a non-scalar orthogonal projection which belongs
to $\mathcal{D}_{\bf v}$. So $\mathcal{D}_{\bf v}$ is not antisymmetric.

Conversely, suppose the nonorthogonality graph associated to ${\bf v}$ is
connected. Let $A \in \mathcal{D}_{\bf v}$ and suppose $A = A^*$. Write
$Av_i = \lambda_iv_i$. Then for any $i$ we have
$$\lambda_i\|v_i\|^2 = \langle Av_i,v_i\rangle = \langle v_i,Av_i\rangle
= \bar{\lambda}_i\|v_i\|^2,$$
so each $\lambda_i$ is real. For any $i$ and $j$ we then have
$$\lambda_i\langle v_i,v_j\rangle = \langle Av_i,v_j\rangle
= \langle v_i, Av_j\rangle = \lambda_j\langle v_i,v_j\rangle,$$
so that $\langle v_i,v_j\rangle \neq 0$ implies $\lambda_i = \lambda_j$.
Since the nonorthogonality graph is connected, this makes every
$\lambda_i$ take the same value $\lambda$, so that $A = \lambda I$. We
have shown that the only self-adjoint elements of $\mathcal{D}_{\bf v}$ are
scalar multiples of the identity, so $\mathcal{D}_{\bf v}$ is antisymmetric.

Now suppose ${\bf v}$ is anti-orthogonal. It is easy to see that the invariant
subspaces for $\mathcal{D}_{\bf v}$ are precisely the subspaces of the form
${\rm span}\{v_i: i \in X\}$ for some $X \subseteq \{1, \ldots, n\}$. Thus the
semi-invariant subspaces are all the orthogonal differences of subspaces of
this form. The structure of the compression of $\mathcal{D}_{\bf v}$ to the
subspace ${\rm span}\{v_i: i \in X\} \ominus {\rm span}\{v_i: i \in Y\}$,
with $X \subset Y$, is clear because ${\rm span}\{v_i: i \in Y\setminus X\}$
is a companion subspace on which $\mathcal{D}_{\bf v}$ acts diagonally.
Thus, the compression of $\mathcal{D}_{\bf v}$ to any semi-invariant
subspace will be a lower dimensional version of $\mathcal{D}_{\bf v}$, and
by anti-orthogonality for a basis whose nonorthogonality graph is complete.
So the compression is antisymmetric by the first part of the theorem. We
have shown that $\mathcal{D}_{\bf v}$ is hereditarily antisymmetric.

Conversely, suppose ${\bf v}$ is not anti-orthogonal. Then there exist
$i$, $j$, and $X \subseteq \{1, \ldots, \hat{i}, \ldots, \hat{j}, \ldots, n\}$
such that $\langle Pv_i,v_j\rangle = 0$, where $P$ is the orthogonal
projection onto the orthocomplement of ${\rm span}\{v_k: k \in X\}$. Then
the orthogonal difference $\mathcal{E}$ between ${\rm span}\{v_k:
k \in X \cup \{i,j\}\}$ and ${\rm span}\{v_k: k \in X\}$ is a two-dimensional
semi-invariant subspace which has ${\rm span}\{v_i,v_j\}$ as a companion
subspace. The compression of $\mathcal{D}_{\bf v}$ to this subspace is then
just the set of $2\times 2$ matrices which are diagonal with respect to the
basis $\{Pv_i, Pv_j\}$ of $\mathbb{C}^2$, which is orthogonal. So this
compression is not antisymmetric, and therefore $\mathcal{D}_{\bf v}$ is not
hereditarily antisymmetric.
\end{proof}

This result shows that there are antisymmetric operator algebras which are
not hereditarily antisymmetric.

\begin{exam}
Let ${\bf v} = \{v_1, v_2, v_3\}$ be a basis for $\mathbb{C}^3$ with the
property that $\langle v_1, v_2\rangle$ and $\langle v_2,v_3\rangle$ are
nonzero but $\langle v_1,v_3\rangle = 0$. Then ${\bf v}$ has a connected
nonorthogonality graph but is not anti-orthogonal, so $\mathcal{D}_{\bf v}$
is antisymmetric but not hereditarily antisymmetric.
\end{exam}

If ${\bf v}$ is anti-orthogonal then $\mathcal{D}_{\bf v}$ is not only
hereditarily antisymmetric, it is maximal for this property. In order to
prove this, we first need to characterize the operator algebras which contain
$\mathcal{D}_{\bf v}$. Fortunately, this characterization is easy and explicit.
The construction was already seen in Example \ref{preorderex}. Given any basis
${\bf v}$ of $\mathbb{C}^n$, let $E_{ij}$ be the operator whose matrix for
the ${\bf v}$ basis has a $1$ in the $(i,j)$ entry and $0$'s elsewhere.
(So $E_{ij}v_j = v_i$ and $E_{ij}v_k = 0$ for $k \neq j$.) 

\begin{theo}\label{precthm}
Let ${\bf v} = \{v_1, \ldots, v_n\}$ be a basis for $\mathbb{C}^n$. For
any preorder $\preceq$ on $\{1, \ldots, n\}$, the set
$$\mathcal{A}_{\preceq} = {\rm span}\{E_{ij}: i \preceq j\}$$
is an operator algebra which contains $\mathcal{D}_{\bf v} =
{\rm span}\{E_{11}, \ldots, E_{nn}\}$. Conversely, every operator algebra in
$M_n$ which contains $\mathcal{D}_{\bf v}$ has this form.
\end{theo}

\begin{proof}
The fact that $\mathcal{A}_{\preceq}$ is an operator algebra is a simple
verification. For the second assertion, let $\mathcal{A} \subseteq M_n$ be
an operator algebra which contains $\mathcal{D}_{\bf v}$. For any
$A \in \mathcal{A}$ we can write $A = \sum a_{ij}E_{ij}$, where $A = (a_{ij})$.
Since $a_{ij}E_{ij} = E_{ii}AE_{jj} \in \mathcal{A}$, it follows that
$\mathcal{A}$ equals ${\rm span}\{E_{ij}: (i,j) \in X\}$ for some subset
$X$ of $\{1, \ldots, n\}^2$, i.e., some relation on $\{1, \ldots, n\}$.
Since each $E_{ii}$ belongs to $\mathcal{D}_{\bf v} \subseteq \mathcal{A}$,
the pairs $(i,i)$ all belong to $X$, i.e., $X$ is reflexive. It is transitive
because $E_{ij}E_{jk} = E_{ik}$. Therefore it is a preorder.
\end{proof}

\begin{theo}
Suppose ${\bf v} = \{v_1, \ldots, v_n\}$ is an anti-orthogonal basis for
$\mathbb{C}^n$. Then $\mathcal{D}_{\bf v}$ is a maximal hereditarily
antisymmetric operator algebra.
\end{theo}

\begin{proof}
Let $\preceq$ be a preorder on $\{1, \ldots, n\}$ and let
$\mathcal{A}_{\preceq}$ be as in Theorem \ref{precthm}. If
$\mathcal{A}_{\preceq}$ properly contains $\mathcal{D}_{\bf v}$ then
there must be at least one pair of distinct elements which are comparable.
Thus, there must exist $1 \leq j \leq n$ for which $X = \{i: i \preceq j\}$
contains at least one number besides $j$. Let $\mathcal{E} =
{\rm span}\{e_i: i \in X\}$; this is an invariant subspace for
$\mathcal{A}_{\preceq}$.

Find a nonzero vector $v \in \mathcal{E}$ which is orthogonal to
${\rm span}\{e_i: i \in X \setminus \{j\}\}$. Write
$v = \sum_{i \in X} a_ie_i$ and consider the operator
$A = \sum_{i \in X} a_iE_{ij} \in \mathcal{A}_{\preceq}$, where $E_{ij}$
is as above. We have $Ae_i = 0$ for all $i \in X \setminus \{j\}$ and
$Av = a_jv$. Thus the compression of $A$ to $\mathcal{E}$ is a nonzero
scalar multiple of the orthogonal
projection onto ${\rm span}\{v\}$, and this shows that $\mathcal{A}_{\preceq}$
is not hereditarily antisymmetric. Since every operator algebra properly
containing $\mathcal{D}_{\bf v}$ has this form, the latter must be maximal
hereditarily antisymmetric.
\end{proof}

The algebras $\mathcal{T}_n$ were maximal hereditarily antisymmetric merely
by virtue of being maximal antisymmetric (Proposition \ref{ex1maxha}). This
is not the case for the algebras $\mathcal{D}_{\bf v}$.

\begin{prop}
Let ${\bf v} = \{v_1, \ldots, v_n\}$ be a basis of $\mathbb{C}^n$ whose
nonorthogonality graph is connected and suppose $n > 2$. Then
$\mathcal{D}_{\bf v}$ is not maximal antisymmetric.
\end{prop}

\begin{proof}
First suppose $\langle v_i,v_j\rangle = 0$ for some $i$ and $j$. Then
$\mathcal{A} = \mathcal{D}_{\bf v} + \mathbb{C}\cdot E_{ij}$ is an algebra
which properly contains $\mathcal{D}_{\bf v}$.
Let $A \in \mathcal{A}$ be self-adjoint. Writing
$A = \lambda_1E_{11} + \cdots + \lambda_nE_{nn} + \alpha E_{ij}$, we have
$$\langle Av_i,v_j\rangle = \lambda_i\langle v_i,v_j\rangle = 0$$
and
$$\langle v_i,Av_j\rangle = \langle v_i, \lambda_j v_j + \alpha v_i\rangle
= \bar{\alpha}\|v_i\|^2,$$
and equating these expressions yields $\alpha = 0$. So
$A \in \mathcal{D}_{\bf v}$, which we know is antisymmetric by Theorem
\ref{ex2thm}, and therefore $A$ must be a scalar multiple of the identity.
So in this case $\mathcal{D}_{\bf v}$ is not maximal antisymmetric.

Otherwise, if no two basis vectors are orthogonal, normalize the basis vectors
so that $\|v_i\| = 1$ for all $i$, and then fix $i$ and $j$ which minimize
$|\langle v_i,v_j\rangle|$. Again let $\mathcal{A} = \mathcal{D}_{\bf v} +
\mathbb{C}\cdot E_{ij}$. Suppose $A \in \mathcal{A}$ is self-adjoint, and
again write $A = \lambda_1E_{11} + \cdots + \lambda_nE_{nn} + \alpha E_{ij}$.
Assume $\alpha \neq 0$, aiming for a contradiction.

By self-adjointness we have
$$\lambda_k = \langle Av_k, v_k\rangle = \langle v_k, Av_k\rangle
= \bar{\lambda}_k$$
for any $k \neq j$, showing that $\lambda_k$ is real for all $k \neq j$.
Then for any $k,l \neq j$
$$\lambda_k\langle v_k,v_l\rangle = \langle Av_k,v_l\rangle
= \langle v_k,Av_l\rangle = \lambda_l\langle v_k,v_l\rangle,$$
and since every $\langle v_k,v_l\rangle$ is nonzero this shows that
$\lambda_k = \lambda_l$. So there is a single value $\lambda$ such that
$\lambda_k = \lambda$ for all $k \neq j$.

We therefore have
$$\lambda\langle v_k,v_j\rangle = \langle Av_k,v_j\rangle
= \langle v_k,Av_j\rangle = \bar{\lambda}_j\langle v_k,v_j\rangle
+ \bar{\alpha}\langle v_k,v_i\rangle$$
for all $k \neq j$. That is,
$$(\lambda - \bar{\lambda}_j)\langle v_k,v_j\rangle =
\bar{\alpha}\langle v_k,v_i\rangle,$$
or more briefly, $\langle v_k,v_i\rangle = \beta\langle v_k,v_j\rangle$
where $\beta = \frac{\lambda - \bar{\lambda}_j}{\bar{\alpha}}$.
With $k = i$ this yields
$$1 = |\beta||\langle v_i,v_j\rangle|,$$
and then minimality of $|\langle v_i,v_j\rangle|$ yields
$$|\langle v_k,v_i\rangle| = |\beta||\langle v_k,v_j\rangle|
\geq |\beta||\langle v_i,v_j\rangle| = 1$$
for all $k$ not equal to $i$ or $j$. (There is at least one such $k$;
this is where we use $n > 2$.) But this is absurd since $\|v_k\| = \|v_i\|
= 1$ and the two vectors are linearly independent. The contradiction shows
that we must have $\alpha = 0$, and thus
$A$ belongs to $\mathcal{D}_{\bf v}$ and hence it must be a scalar multiple
of the identity by Theorem \ref{ex2thm}. This completes the proof that
$\mathcal{A}$ is antisymmetric, so we conclude that $\mathcal{D}_{\bf v}$
can always be strictly enlarged to an algebra which is still antisymmetric.
\end{proof}

If $n = 2$ then $\mathcal{D}_{\bf v}$ is maximal antisymmetric for any
nonorthogonal basis ${\bf v} = \{v_1,v_2\}$ of $\mathbb{C}^2$. This is
because its dimension is $2$, and we know from Proposition \ref{dimcount}
that no operator algebra contained in $M_2$ whose dimension is at least $3$
can be antisymmetric.

One can ask, for which preorders $\preceq$ is the algebra
$\mathcal{A}_{\preceq}$ antisymmetric? It seems natural to conjecture that
in order for this to be the case, $\preceq$ would have to be a genuine partial
order. However, this is not correct.

\begin{exam}\label{fullsubex}
Let ${\bf v} = \{v_1, v_2, v_3, v_4\}$ be a basis for $\mathbb{C}^4$
satisfying $\langle v_i, v_j\rangle \neq 0$ for all $i$ and $j$,
and assume the additional condition that
no nonzero vector in $\mathcal{E} = {\rm span}\{v_1, v_2\}$ is orthogonal to
every vector in $\mathcal{F} = {\rm span}\{v_3, v_4\}$, i.e., we have
$\mathcal{E} \cap \mathcal{F}^\perp = \{0\}$. Define $\mathcal{A}$ to be
$\mathcal{D}_{\bf v} + {\rm span}\{E_{12}, E_{21}\}$. This is the algebra
$\mathcal{A}_{\preceq}$ from Theorem \ref{precthm} for the
preorder which sets $1 < 2$ and $2 < 1$, with no other comparability.

$\mathcal{A}$ is antisymmetric. To see this, suppose $A = \lambda_1E_{11} +
\lambda_2E_{22} + \lambda_3E_{33} + \lambda_4E_{44} + \alpha E_{12} +
\beta E_{21}$ is self-adjoint. As usual, the computation
$$\lambda_4\|v_4\|^2 = \langle Av_4, v_4\rangle = \langle v_4, Av_4\rangle =
\bar{\lambda}_4\|v_4\|^2$$
implies that $\lambda_4$ is real and then
$$\lambda_3\langle v_3, v_4\rangle = \langle Av_3, v_4\rangle =
\langle v_3, Av_4\rangle = \lambda_4\langle v_3, v_4\rangle$$
shows that $\lambda_3 = \lambda_4$. Let $B = A - \lambda_4I$; we must show
that $B = 0$.

But this is easy because $B$ is self-adjoint and therefore
${\rm ran}(B) = {\rm ker}(B)^\perp$. The kernel of $B$ contains $\mathcal{F}$,
so this shows that its range is contained in $\mathcal{F}^\perp$. But
it is also contained in $\mathcal{E}$. The hypothesis that
$\mathcal{E} \cap \mathcal{F}^\perp = \{0\}$ then implies that $B = 0$.
\end{exam}

At the same time, $\preceq$ can be a partial order without
$\mathcal{A}_\preceq$ being antisymmetric. For instance, if $\preceq$ is
a linear order then $\mathcal{A}_\preceq$ consists of the upper triangular
operators for some ordering of the ${\bf v}$ basis, and it cannot be
antisymmetric because its dimension is too large (Proposition \ref{dimcount}).

\section{Finite dimensional structure analysis}

We are ready to discuss the general structure of hereditarily antisymmetric
operator algebras in $M_n$. We start with the fact that they can always be
upper triangularized. This will be an immediate corollary of the following
theorem, which characterizes the operator algebras in finite dimensions
which can be upper triangularized.

Recall that a subquotient of an operator algebra $\mathcal{A}
\subseteq M_n$ is an operator algebra of the form
$P\mathcal{A}P \subseteq B(\mathcal{E})$ where $P$ is the orthogonal
projection onto a semi-invariant subspace $\mathcal{E}$. Say that this
subquotient is {\it full} if it equals $B(\mathcal{E})$. 

\begin{theo}\label{utthm}
Let $\mathcal{A} \subseteq M_n$ be an operator algebra. The following are
equivalent:

{\narrower{
\noindent (i) there is an ordered orthonormal basis of $\mathbb{C}^n$ with
respect to which every operator in $\mathcal{A}$ has an upper triangular matrix

\noindent (ii) there is an ordered vector space basis of $\mathbb{C}^n$ with
respect to which every operator in $\mathcal{A}$ has an upper triangular matrix

\noindent (iii) $\mathcal{A}$ has no full subquotients of dimension
greater than $1$.

}}
\end{theo}

\begin{proof}
The equivalence of (i) and (ii) was Proposition \ref{orthout}. The proof of
(iii) $\Rightarrow$ (i) goes by induction on $n$. It is trivial for $n = 1$.
For $n > 1$, assume $\mathcal{A}$ has no full subquotients of
dimension greater than $1$. It must be properly contained in $M_n$ as
otherwise it would be a full subquotient of itself. Therefore
Burnside's theorem yields that it has a nontrivial invariant subspace
$\mathcal{E}$. Both $\mathcal{E}$ and $\mathcal{E}^\perp$ are then
semi-invariant, and letting $P$ and $Q$ be the orthogonal projections
onto these two subspaces, Proposition \ref{compofcomp} yields that both
$P\mathcal{A}P$ and $Q\mathcal{A}Q$ satisfy condition (iii),
so inductively we can find ordered orthonormal bases of $\mathcal{E}$ and
$\mathcal{E}^\perp$ with respect to which every operator in $P\mathcal{A}P$
and $Q\mathcal{A}Q$ has an upper triangular matrix. The concatenation of
these two bases is then an ordered orthonormal basis of $\mathbb{C}^n$ with
respect to which every operator in $\mathcal{A}$ has an upper triangular matrix.

For (i) $\Rightarrow$ (iii), suppose there is an ordered orthonormal basis
$\{f_1, \ldots, f_n\}$ of $\mathbb{C}^n$ with respect to which every operator
in $\mathcal{A}$ has an upper triangular matrix. To reach a contradiction,
assume $\mathcal{A}$ has a full subquotient of dimension greater
than $1$. Consider first the case that the corresponding semi-invariant
subspace $\mathcal{E}$ is actually invariant.

Let $v,w \in \mathcal{E}$ be linearly independent. Without loss of generality
we can assume that there is an index $1 \leq k \leq n$ with the property that
$\langle v, f_k\rangle \neq 0$ and $\langle v, f_{k'}\rangle =
\langle w,f_{k'}\rangle = 0$ for all $k' > k$. Now since $\mathcal{E}$ is
invariant and $P\mathcal{A}P$ is full, there are operators $A,B \in
\mathcal{A}$ satisfying $Av = v$, $Aw = 0$ and $Bv = 0$, $Bw = v$. Since $A$
is upper triangular for the $(f_i)$ basis and $v$ has no components past
$k$, the condition $Av = v$ implies that the $(k,k)$ entry of $A$ is $1$.
Since $w$ also has no nonzero components past $k$, $Aw = 0$ then implies
that $\langle w, f_k\rangle = 0$. But now $B$ being upper triangular
implies that $\langle Bw,f_k\rangle = 0$, which makes $Bw = v$ impossible.
This contradiction shows that the case where $\mathcal{E}$ is invariant
cannot happen.

Now consider the general case where $\mathcal{E}$ is merely semi-invariant.
Let $\mathcal{F} \subseteq \mathcal{E}^\perp$ be an invariant subspace such
that $\mathcal{F} \oplus \mathcal{E}$ is also invariant and let $Q$ be the
orthogonal projection onto $\mathcal{F}^\perp$. Then the vectors $Qf_1,
\ldots, Qf_n$ span $\mathcal{F}^\perp$, so by removing each $Qf_i$ which is
a linear combination of $Qf_1, \ldots, Qf_{i-1}$ we get an ordered basis
$(w_j)$ of $\mathcal{F}^\perp$. Now any $A \in M_n$ which is upper triangular
for $(f_i)$ and for which $\mathcal{F}$ is invariant will satisfy
$A(Qf_i - f_i) \in \mathcal{F}$ since $Qf_i - f_i \in \mathcal{F}$, and hence
$QA(Qf_i - f_i) = 0$. Thus
$$QAQ(Qf_i) = QA(Qf_i - f_i) + QAf_i = QAf_i \in
{\rm span}\{Qf_1, \ldots, Qf_i\},$$
i.e., $QAQ$ will be upper triangular for the $(w_j)$ basis of
$\mathcal{F}^\perp$. We have shown that there is an ordered vector space basis
of $\mathcal{F}^\perp$ with respect to which every operator in $Q\mathcal{A}Q$
has an upper triangular matrix, and (invoking Proposition \ref{orthout})
this reduces us to the first case
because $Q\mathcal{A}Q$ still has $P\mathcal{A}P$ as a subquotient
and $\mathcal{E}$ is invariant for $Q\mathcal{A}Q$. This completes the proof.
\end{proof}

Several characterizations of upper triangularizability of algebras of matrices
can be found in \cite{RR}, and the implication (iii) $\Rightarrow$(i) in the
preceding theorem can be inferred from Lemma 1.1.4 of \cite{RR}. However,
as far as I know this implication has not been explicitly stated anywhere,
and the reverse implication (i) $\Rightarrow$ (iii)) seems to be entirely new;
compare Example \ref{iutex}, which shows that it fails in infinite dimensions.
(Note that the ``quotients'' of \cite{RR} are what I am calling
``subquotients''.)

\begin{coro}\label{utcor}
Let $\mathcal{A} \subseteq M_n$ be hereditarily antisymmetric. Then there
is an ordered orthonormal basis of $\mathbb{C}^n$ with respect to which every
operator in $\mathcal{A}$ has an upper triangular matrix.
\end{coro}

Mere antisymmetry is not a sufficient hypothesis to ensure this conclusion.
We already saw, in Example \ref{fullsubex}, an antisymmetric operator algebra
in $M_4$ which had an invariant subspace $\mathcal{E}$ of dimension $2$ such
that the corresponding subobject was full. Thus, according
to Theorem \ref{utthm} this algebra cannot be made upper triangular.

\begin{coro}
The antisymmetric operator algebra of Example \ref{fullsubex} cannot be upper
triangularized.
\end{coro}

Let us also note that $\mathcal{A}$ being upper triangular does not prevent
the existence of an orthogonal projection $P$ onto a subspace $\mathcal{E}$
of dimension greater than $1$ such that $P\mathcal{A}P \cong B(\mathcal{E})$.
For instance, according to Proposition 2.2 of \cite{W2},
the algebra $\mathcal{D}$ of diagonal matrices in $M_{k^2 + k - 1}$ satisfies
$P\mathcal{D}P \cong M_k$ for some orthogonal projection $P$ whose range has
dimension $k$. Thus, the obstruction to upper triangularizing is not the
existence of a $k$-clique, in the terminology of \cite{W2}, but the existence
of a $k$-clique which lives on a semi-invariant subspace.

Our structure theorem for hereditarily antisymmetric operator algebras in
$M_n$ will state that every such algebra is contained in the algebra of all
Jordanesque matrices relative to some basis of $\mathbb{C}^n$. Let us
introduce a notation for this algebra.

\begin{defi}
Let a {\it block ordered basis} of $\mathbb{C}^n$ be an ordered basis
$(v_1, \ldots, v_n)$ together with a choice of $n_1, \ldots, n_k > 0$
satisfying $n_1 + \cdots + n_k = n$. Given a block ordered basis ${\bf v}$,
define $\mathcal{J}_{\bf v}$ to be the set of all operators on $\mathbb{C}^n$
whose matrices for ${\bf v}$ are $(n_1, \ldots, n_k)$-Jordanesque .
\end{defi}

It will often be convenient to indicate the blocks explicitly in the
listing of basis vectors, i.e., as $(v_1^1, \ldots, v_{n_1}^1, \ldots,
v_1^k, \ldots, v_{n_k}^k)$.
As a matter of normalization, we can assume that for each $1 \leq j \leq k$
the set $\{v^j_1, \ldots, v^j_{n_j}\}$ is orthonormal. Applying Proposition
\ref{orthout} to the span of each of these sets shows that the set of all
Jordanesque matrices for any given basis equals the set of all Jordanesque
matrices for a basis which is normalized in this way.

Depending on the nature of the block ordered basis ${\bf v}$, the algebra
$\mathcal{J}_{\bf v}$ may or may not be hereditarily antisymmetric.
The invariant subspaces for $\mathcal{J}_{\bf v}$ are precisely the subspaces
of the form ${\rm span}\{v^1_1, \ldots, v^1_{m_1}, \ldots, v^k_1, \ldots,
v^k_{m_k}\}$ where $0 \leq m_j \leq n_j$ for each $j$. We allow $m_j = 0$
to accomodate the possibility that no vectors from the $j$th block appear.
Call this subspace the {\it $(m_1, \ldots, m_k)$ subspace.} The
condition we need to ensure hereditary antisymmetry is the following.

\begin{defi}
Say that a block ordered basis ${\bf v} =
(v^1_1, \ldots, v^1_{n_1}, \ldots, v^k_1, \ldots, v^k_{n_k})$ is
{\it normalized} if for each $j$ the set $\{v^j_1, \ldots, v^j_{n_j}\}$ is
orthonormal. Say that it is {\it suitably nonorthogonal} if, for every
$0 \leq m_1 \leq n_1$, $\ldots$, $0 \leq m_k \leq n_k$ and every
$1 \leq j < j' \leq k$ such that $m_j < n_j$ and $m_{j'} < n_{j'}$, we have
$$\langle Pv^j_{m_j+1}, v^{j'}_{m_{j'} + 1}\rangle \neq 0,\eqno{(*)}$$
where $P$ is the orthogonal projection onto the orthocomplement of the
$(m_1, \ldots, m_k)$ subspace.
\end{defi}

Thus, in order to be suitably nonorthogonal, the basis must be orthonormal
in each block and a finite number of additional conditions ($*$) must be
satisfied. For instance, taking $m_1 = \cdots = m_k = 0$ in ($*$) shows that we
require $\langle v^j_1, v^{j'}_1\rangle \neq 0$ for all $1 \leq j < j' \leq k$.
Letting exactly one $m_j$ be nonzero, it is not too hard to see that ($*$)
reduces to the requirement that $\langle v^j_i, v^{j'}_1\rangle \neq 0$ for
all $j \neq j'$ and $1 \leq i \leq n_j$ (since the $(m_1, \ldots, m_k)$
subspace in this case is just ${\rm span}\{v^j_1, \ldots, v^j_{i-1}\}$,
to which $v^j_i$ is already orthogonal). If more than one of the
$m_j$ are nonzero then the condition becomes more complicated,
but it should be generically satisfied: the bases for which
$Pv^j_{m_j+1}$ and $Pv^{j'}_{m_{j'} + 1}$ are orthogonal are exceptional.

Say that an operator algebra $\mathcal{A} \subseteq \mathcal{J}_{\bf v}$
{\it distinguishes blocks} if for any $1 \leq j < j' \leq k$, there is an
operator in $\mathcal{A}$ whose matrix has different values on the main
diagonals of the $j$ and $j'$ blocks.

\begin{theo}\label{dvthm}
Let ${\bf v} = \{v^1_1, \ldots, v^1_{n_1}, \ldots, v^k_1, \ldots, v^k_{n_k}\}$
be a normalized block ordered basis of $\mathbb{C}^n$. If ${\bf v}$ is suitably
nonorthogonal then $\mathcal{J}_{\bf v}$ is hereditarily antisymmetric. If
${\bf v}$ is not suitably nonorthogonal then $\mathcal{J}_{\bf v}$
is not hereditarily antisymmetric, nor is any subalgebra of
$\mathcal{J}_{\bf v}$ which distinguishes blocks.
\end{theo}

\begin{proof}
Suppose ${\bf v}$ is suitably nonorthogonal. We must show that if
$\mathcal{E}$ is semi-invariant for $\mathcal{J}_{\bf v}$ and
$PAP \in B(\mathcal{E})$ is self-adjoint, where $P$ is the orthogonal
projection onto $\mathcal{E}$ and $A \in \mathcal{J}_{\bf v}$, then
$PAP$ is a scalar multiple of $P$.

Fix $\mathcal{E}$ and $A$ such that $PAP$ is self-adjoint. Now
$\mathcal{E} = \mathcal{E}_1 \ominus \mathcal{E}_2$ where $\mathcal{E}_1$
is the $(s_1, \ldots, s_k)$ subspace and $\mathcal{E}_2$ is the
$(r_1, \ldots, r_k)$ subspace, for some $0 \leq r_1 \leq s_1 \leq n_1$,
$\ldots$, $0 \leq r_k \leq s_k \leq n_k$. The basis vectors
$v^1_{r_1 + 1}, \ldots, v^1_{s_1}, \ldots, v^k_{r_k + 1}, \ldots,
v^k_{s_k}$ need not belong to $\mathcal{E}$, but they span a companion
subspace $\mathcal{F}$ and their orthogonal projections into $\mathcal{E}$
constitute a basis for $\mathcal{E}$. With respect to this basis $PAP$ is
Jordanesque. This follows from the fact that $QAQ$ is Jordanesque, where $Q$
is the natural projection onto $\mathcal{F}$, together with Proposition
\ref{compprop}.

For each $j$, consider the orthogonal projection $P_j$ onto
${\rm span}\{Pv^j_{r_j + 1}, \ldots, Pv^j_{s_j}\}$. The matrix of $P_jAP_j
= P_j(PAP)P_j$ for this basis is upper triangular and constant on the main
diagonal, but it is also self-adjoint, and thus it follows from Propositions
\ref{orthout} and \ref{ex1maxha} that $P_jAP_j$ is a scalar multiple of $P_j$.

Since this was true for all $j$, it follows that $PAP$ is diagonal for
the basis of $\mathcal{E}$ obtained by projecting the standard basis of
$\mathcal{F}$ into $\mathcal{E}$.
Say $P_jAP_j = \lambda_j P_j$, i.e., $\lambda_j$ is the main diagonal
entry of $A$ on the $j$th block of ${\bf v}$. Fix distinct indices $j, j'$
for which $r_j < s_j$ and $r_{j'} < s_{j'}$. Suitable nonorthogonality of
${\bf v}$ then implies that $Pv^j_{r_j + 1}$ and $Pv^{j'}_{r_{j'}+1}$ are
not orthogonal. Since these are eigenvectors belonging to the eigenvalues
$\lambda_j$ and $\lambda_{j'}$, the usual computation (as in the proof of
Theorem \ref{ex2thm}, for example) then shows that self-adjointness of $PAP$
implies $\lambda_j = \lambda_{j'}$. As $j$ and $j'$ were arbitrary (among
the $j$'s represented in $\mathcal{E}$), we
conclude that $PAP$ is a scalar multiple of $P$, as desired. This completes
the proof of the first assertion of the theorem.

For the second assertion, suppose ${\bf v}$ is not suitably nonorthogonal;
we must show that every subalgebra $\mathcal{A}$ of $\mathcal{J}_{\bf v}$
which distinguishes blocks is not hereditarily antisymmetric. (Obviously,
this includes $\mathcal{J}_{\bf v}$ itself.) By Proposition
\ref{unitprop} we can restrict attention to unital subalgebras. By Corollary
\ref{bigdiag} every such algebra contains the algebra $\mathcal{A}_{\rm diag}$
of operators whose matrix for the ${\bf v}$ basis is diagonal, with constant
main diagonal entries on each block. So we must find such a matrix whose
compression to some semi-invariant subspace (for $\mathcal{A}$) is
nonscalar and self-adjoint.

Since ${\bf v}$ fails to be suitably nonorthogonal, there exist
$m_1, \ldots, m_k$ and $1 \leq j < j' \leq k$ such that
$Pv^j_{m_j+1}$ and $Pv^{j'}_{m_{j'} + 1}$ are orthogonal, where $P$ is
the orthogonal projection onto the orthocomplement of the
$(m_1, \ldots, m_k)$ subspace. Let $m'_l = m_l$ when $l \not\in
\{j, j'\}$, $m'_j = m_j + 1$, and $m'_{j'} = m_{j'} + 1$ and consider
the orthogonal difference of the $(m'_1, \ldots, m'_k)$ subspace and
the $(m_1, \ldots, m_k)$ subspace. These subspaces are invariant for
$\mathcal{J}_{\bf v}$, and hence also for $\mathcal{A}$, so their orthogonal
difference is semi-invariant for $\mathcal{A}$. It is spanned by the
orthogonal vectors $Pv^j_{m_j+1}$ and $Pv^{j'}_{m_{j'} + 1}$, and there
is an operator in $\mathcal{A}_{\rm diag}$ whose
$j$ and $j'$ eigenvalues are distinct and real, and thus whose compression to
$\mathcal{E}$ is nonscalar and self-adjoint. We conclude that $\mathcal{A}$
is not hereditarily antisymmetric.
\end{proof}

Theorem \ref{dvthm} tells us that if $\mathcal{J}_{\bf v}$ fails to be
hereditarily antisymmetric then so does every subalgebra which distinguishes
blocks. It does not tell us that if $\mathcal{J}_{\bf v}$ is hereditarily
antisymmetric then the same is true of any subalgebra which distinguishes
blocks. In fact, this can fail.

\begin{exam}
Let ${\bf v} = \{v^1_1, v^1_2, v^2_1, v^2_2\}$ be a normalized suitably
nonorthogonal basis of $\mathbb{C}^4$ satisfying
$\langle v^1_2, v^2_2\rangle = 0$. For bases of this size, suitable
nonorthogonality amounts to the sets $\{v^1_1, v^1_2\}$
and $\{v^2_1, v^2_2\}$ both being orthonormal, $v^1_1$ not being orthogonal
to either $v^2_1$ or $v^2_2$, $v^2_1$ not being orthogonal to either
$v^1_1$ or $v^1_2$, and $Pv^1_2$ and $Pv^2_2$ not being orthogonal, where $P$
is the orthogonal projection onto the orthocomplement of
${\rm span}\{v^1_1, v^2_1\}$. This is compatible with $v^1_2$ and $v^2_2$
being orthogonal.

According to Theorem \ref{dvthm}, $\mathcal{J}_{\bf v}$ is hereditarily
antisymmetric. However, the elements of $\mathcal{J}_{\bf v}$ which are
diagonal for ${\bf v}$ has ${\rm span}\{v^1_2,v^2_2\}$ as an invariant
subspace, and the elements of $\mathcal{J}_{\bf v}$ which are diagonal
for ${\bf v}$ compress to all operators on this subspace which are diagonal
for the orthonormal basis $\{v^1_2, v^2_2\}$. So these elements constitute
a subalgebra of $\mathcal{J}_{\bf v}$ which distinguishes blocks but is
not hereditarily antisymmetric.
\end{exam}

Of course, if $\mathcal{J}_{\bf v}$ is hereditarily antisymmetric then in
particular it is antisymmetric, and hence so is any subalgebra. So every
subalgebra is antisymmetric, but not necessarily hereditarily antisymmetric.

Before proving our structure theorem for hereditarily antisymmetric
operator algebras, we need one more easy lemma.

\begin{lemma}\label{stlemma}
Let $\mathcal{A} \subseteq \widetilde{M}_n$ be an operator algebra all of
whose elements are upper triangular. Then there exists an operator
$A \in \mathcal{A}$
all of whose main diagonal entries are real and with the property that
$a_{ii} = a_{jj}$ implies $b_{ii} = b_{jj}$ for all $B \in \mathcal{A}$,
where $A = (a_{ij})$ and $B = (b_{ij})$.
\end{lemma}

\begin{proof}
The main diagonal entries of the matrices in $\mathcal{A}$ constitute a
subalgebra of $l^\infty_n$. By Lemma \ref{linflemma} and the comment
following it, there is an equivalence relation on $\{1, \ldots, n\}$ such
that this subalgebra consists of all the functions which are constant on
each block (if $\mathcal{A}$ is unital) or all the functions which are
constant on each block and vanish on
some specified block (if it is not). In either case we can find a function
which takes a different real value on each block, and then take $A$ to be
a matrix in $\mathcal{A}$ whose main diagonal entries are this function.
\end{proof}

\begin{theo}\label{mainthm}
Let $\mathcal{A} \subseteq M_n$ be a hereditarily antisymmetric operator
algebra. Then there is a normalized suitably nonorthogonal block ordered
basis $(v^1_1, \ldots, v^1_{n_1}, \ldots, v^k_1, \ldots, v^k_{n_k})$ of
$\mathbb{C}^n$, with respect to which the matrix of every
element of $\mathcal{A}$ is Jordanesque.
\end{theo}

\begin{proof}
We will prove that there are numbers $n_1 + \cdots + n_k = n$ and a block
ordered basis
$(v^1_1, \ldots, v^1_{n_1}, \ldots, v^k_1, \ldots, v^k_{n_k})$ with respect
to which the matrix of every element of $\mathcal{A}$ is Jordanesque and
such that $\mathcal{A}$ distinguishes blocks. After orthonormalizing each
block, we get a basis which must be suitably nonorthogonal by the second
part of Theorem \ref{dvthm}.

We can assume that $\mathcal{A}$ is unital by Proposition \ref{unitprop}.
The proof goes by induction on $n$. First,
according to Corollary \ref{utcor} we can find an ordered orthonormal basis
${\bf w} = (w_1, \ldots, w_n)$ of $\mathbb{C}^n$ with respect to
which the matrix of every operator in $\mathcal{A}$ is upper triangular.
Then $\mathcal{E}_0 = {\rm span}\{w_1, \ldots, w_{n-1}\}$ is an invariant
subspace, so we can inductively assume that $\mathcal{E}_0$ has a block
ordered basis ${\bf v}_0$ with respect to which the compression of
$\mathcal{A}$ to $\mathcal{E}_0$ distinguishes blocks and every element
of which is Jordanesque.

Fix $A \in \mathcal{A}$ as in Lemma \ref{stlemma}, relative to the ${\bf w}$
basis. Now the $(n,n)$ entry of $A$ for the ${\bf w}$ basis may equal at
least one of its other main diagonal entries. If $\lambda$ is this bottom
right entry of $A$, then this means that the entries of $A$ in the main
diagonal of exactly one of the blocks of the ${\bf v}_0$ basis are equal to
$\lambda$. The other possibility is that $\lambda$ is distinct from all other
main diagonal entries (i.e., eigenvalues) of $A$.

In any case, since distinct blocks of ${\bf v}_0$ can be interchanged without
consequence, we can assume that ${\bf v}_0 = (v^1_1, \ldots, v^1_{n_1},
\ldots, v^k_1, \ldots, v^k_{n_k-1})$ where $n_1 + \cdots + n_k = n$ and
the main diagonal entries of $A$ in the final block are all $\lambda$. The
case where $\lambda$ is distinct from the other eigenvalues of $A$ is
accommodated by allowing the possibility that $n_k = 1$.

Next, let $\mathcal{F}$ be the span of
${\bf v}_1 = (v^1_1, \ldots, v^1_{n_1}, \ldots, v^{k-1}_1, \ldots,
v^{k-1}_{n_{k-1}}, w_n)$ (omitting the $k$th block)
and let $P$ be the natural projection onto $\mathcal{F}$ with kernel
${\rm span}\{v^k_1, \ldots, v^k_{n_k - 1}\}$. Then $PAP \in B(\mathcal{F})$
has an upper triangular matrix for the ${\bf v}_1$ basis, and its bottom right
entry is $\lambda$, and no other main diagonal entry takes this value. Thus
$\lambda$ is an eigenvalue of $PAP$. Let $v^k_{n_k} \in \mathcal{F}$ be an
eigenvector for this eigenvalue, so that $PAv^k_{n_k} = \lambda v^k_{n_k}$.
This implies that
$Av^k_{n_k} \in {\rm span}\{v^k_1, \ldots, v^k_{n_k}\}$, so the matrix
of $A$ is Jordanesque with respect to the basis ${\bf v} =
(v^1_1, \ldots, v^1_{n_1}, \ldots, v^k_1, \ldots, v^k_{n_k})$.

We must show that the matrix of every $B \in \mathcal{A}$ is Jordanesque
with respect to the ${\bf v}$ basis. That is, the matrix of $B$ with respect
to this basis must be zero in all but the last $n_k$ entries of its final
column. (We already have, inductively, that the upper left $(n-1)\times(n-1)$
corner of this matrix is $(n_1, \ldots, n_{k-1})$-Jordanesque.)

Suppose this fails, and find $B \in \mathcal{A}$ whose final column
$Bv^k_{n_k}$ has a nonzero $v^j_i$ component for some $j < k$, but such that
it and all other operators in $\mathcal{A}$ have zero components in the
$v^j_{i+1}, \ldots, v^j_{n_j}, \ldots, v^{k-1}_1, \ldots, v^{k-1}_{n_{k-1}}$
entries. That is, $v^j_i$ is the highest index where Jordanesqueness of some
operator in $\mathcal{A}$ fails. Then $\mathcal{E}_1 =
{\rm span}\{v^1_1, \ldots, v^1_{n_1}, \ldots, v^j_1, \ldots,
v^j_i, v^k_1, \ldots, v^k_{n_k}\}$ (i.e., ${\bf v}$ with all entries between
$v^j_i$ and $v^k_1$ omitted) and $\mathcal{E}_2 =
{\rm span}\{v^1_1, \ldots, v^1_{n_1}, \ldots, v^j_1, \ldots, v^j_{i-1},
v^k_1, \ldots, v^k_{n_k-1}\}$ (i.e., the same list but also omitting $v^j_i$
and $v^k_{n_k}$) are both invariant for $\mathcal{A}$, and so their orthogonal
difference $\mathcal{E} = \mathcal{E}_1 \ominus \mathcal{E}_2$ is
semi-invariant. This subspace is two-dimensional and
${\rm span}\{v^j_i, v^k_{n_k}\}$ is a companion subspace.
Let $Q$ be the natural projection onto ${\rm span}\{v^j_i, v^k_{n_k}\}$.
Then $QAQ$ is diagonal, with
distinct real main diagonal entries, for the $\{v^j_i, v^k_{n_k}\}$ basis;
$QIQ$ is diagonal for the same basis with diagonal entries $1$ and $1$; and
$QBQ$ has a nonzero entry in the $(1,2)$ corner. So $Q\mathcal{A}Q$ has
dimension at least $3$, which by Proposition \ref{compprop}
means that the compression of $\mathcal{A}$ to
$\mathcal{E}$ has dimension at least $3$, and it is therefore not antisymmetric
by Proposition \ref{dimcount}. This contradicts hereditary antisymmetry of
$\mathcal{A}$, and we conclude that the matrix of every operator in
$\mathcal{A}$ for the ${\bf v}$ basis must be Jordanesque.
\end{proof}

Thus, Corollary \ref{bigdiag} applies to any hereditarily antisymmetric
operator algebra in $M_n$.

The hypothesis of Theorem \ref{utthm} (iii) does not suffice to imply the
conclusion that $\mathcal{A}$ can be made Jordanesque. For example, the
set of all upper triangular matrices in $M_n$ has no full subquotients
of dimension greater than $1$ (by Theorem \ref{utthm}, or by
inspection), yet it cannot be put in Jordanesque form: its dimension is
greater than the dimension of any $\mathcal{J}_{\bf v}$ in $M_n$.

We have a rather explicit characterization of the maximal hereditarily
antisymmetric operator algebras.

\begin{coro}
The maximal hereditarily antisymmetric subalgebras of $M_n$ are precisely
the algebras $\mathcal{J}_{\bf v}$ for ${\bf v}$ a normalized suitably
nonorthogonal block ordered basis of $\mathbb{C}^n$.
\end{coro}

This follows from Theorems \ref{mainthm} and \ref{dvthm}: the former
shows that every hereditarily antisymmetric algebra is contained in such
an algebra, and the latter shows that every such algebra is hereditarily
antisymmetric.

\section{``Quantum'' posets}

I am proposing that unital hereditarily antisymmetric operator algebras may
be regarded as ``quantum posets''. The qualifier ``quantum'' is justified
both by the direct physical interpretation discussed in Section 1, and
on the grounds of kinship with other objects, such as quantum graphs and
quantum metrics, which have similar physical interpretations \cite{DSW, KW}.

As I mentioned in the introduction, the nilpotent part of a hereditarily
antisymmetric operator algebra (see Corollary \ref{bigdiag} plus Theorem
\ref{mainthm}) plays the role of a strict order, i.e., it is the quantum
version of $\prec$ rather than $\preceq$. Since every {\it nilpotent operator
algebra}  --- every operator algebra consisting solely of nilpotent matrices
--- is hereditarily antisymmetric, I am also proposing that, in finite
dimensions, nilpotent operator algebras are ``strict quantum orders''.

In order to give this proposal substance, we need some nontrivial results
about operator algebras which are analogous to known
results about posets. In this section I will prove operator algebraic analogs
of the theorems of Mirsky (the maximal length of a chain equals the minimal
size of a decomposition into antichains) and Dilworth (the maximal width of an
antichain equals the minimal size of a decomposition into chains) for posets.

First, we need to identify operator algebraic analogs of chains and antichains.

\begin{defi}
A {\it quantum antichain} for a nilpotent operator algebra $\mathcal{A}
\subset M_n$ is a nonzero semi-invariant subspace $\mathcal{E} \subseteq
\mathbb{C}^n$ whose corresponding subquotient $P\mathcal{A}P$ equals $\{0\}$.
Its {\it width} is the dimension of $\mathcal{E}$. A {\it quantum chain} for
$\mathcal{A}$ is a sequence of nonzero vectors $\mathcal{C} =
(v_1, \ldots, v_k)$ in $\mathbb{C}^n$ with $v_{i+1} \in \mathcal{A}v_i$ for
$1 \leq i < k$. Its {\it length} is $k$.
\end{defi}

In the definition of quantum chains, we want $\mathcal{A}$ not to contain
any nonzero projections, i.e., to be nilpotent, in order to express the
idea that chains are strictly descending, not merely descending.
Note that since $\mathcal{A}$ is an algebra, we actually get
$v_j \in \mathcal{A}v_i$ whenever $i < j$.

It is convenient to know that in the definition of quantum antichains,
semi-invariance of $\mathcal{E}$ is automatic.

\begin{prop}\label{acprop2}
Let $\mathcal{A} \subset M_n$ be a nilpotent operator algebra, let
$\mathcal{E}$ be a subspace of $\mathbb{C}^n$, and let $P$ be the orthogonal
projection onto $\mathcal{E}$. If $P\mathcal{A}P = \{0\}$ then $\mathcal{E}$
is a quantum antichain for $\mathcal{A}$.
\end{prop}

\begin{proof}

This can be verified directly, but the quick way to see it is to invoke
Theorem 2.16 of \cite{D}, which states that $\mathcal{E}$ is semi-invariant
if and only if the map $A \mapsto PAP$ is a homomorphism from $\mathcal{A}$
to $B(\mathcal{E})$. If $P\mathcal{A}P = \{0\}$ then this map must be
the zero homomorphism.
\end{proof}

Next, we need matrix versions of the idea of partitioning a poset into chains
or antichains.

\begin{defi}
Let $\mathcal{A} \subset M_n$ be a nilpotent operator algebra.
\smallskip

\noindent (a) A {\it partition into quantum antichains} is an orthogonal
decomposition $\mathbb{C}^n = \mathcal{E}_1 \oplus \cdots \oplus \mathcal{E}_k$
where each $\mathcal{E}_i$ is a quantum antichain. Its {\it size} is $k$. It
is {\it ordered} if $\mathcal{E}_1 \oplus \cdots \oplus \mathcal{E}_i$ is
invariant for each $1 \leq i \leq k$.
\smallskip

\noindent (b) A family of quantum chains $\mathcal{C}_1, \ldots, \mathcal{C}_k$
{\it spans} $\mathbb{C}^n$ if the span of $\bigcup_{i=1}^k \mathcal{C}_i$
equals $\mathbb{C}^n$. A {\it partition into quantum chains} is a family of
quantum chains for which this union constitutes a basis for $\mathbb{C}^n$.
Its {\it size} is $k$.
\end{defi}

Intuitively, a partition into quantum antichains is ordered if $\mathcal{E}_i$
is ``below'' $\mathcal{E}_j$ whenever $i < j$.
Classically, given two disjoint antichains $\mathcal{C}_1$ and $\mathcal{C}_2$
in a poset, we can always perform a swap: let $\mathcal{C}_1'$ be the set of
$x \in \mathcal{C}_1$ which lie above some element of $\mathcal{C}_2$, let
$\mathcal{C}_2'$ be the set of $y \in \mathcal{C}_2$ which lie below some
element of $\mathcal{C}_1$, and replace $\mathcal{C}_1$ and $\mathcal{C}_2$
with the antichains $(\mathcal{C}_1\setminus\mathcal{C}_1') \cup
\mathcal{C}_2'$ and $(\mathcal{C}_2\setminus \mathcal{C}_2')\cup
\mathcal{C}_1'$. Then no element of the first new antichain lies above
any element of the second new antichain. Using this trick repeatedly, any
partition into antichains can classically be converted into an ordered
partition without changing its size. However, nothing like this is true in
the quantum setting; compare Theorem \ref{qmir} and Example \ref{nomir} below.

Any family of quantum chains which spans $\mathbb{C}^n$ can be turned into a
partition by removing selected elements. This follows from the fact that any
subset of a quantum chain is a quantum chain --- this is a consequence of the
earlier comment that $v_j \in \mathcal{A}v_i$ whenever $i < j$ ---
so we can simply remove excess elements until there is no linear dependence.
I record this fact:

\begin{prop}\label{chainspan}
Let $\mathcal{A} \subset M_n$ be a nilpotent
operator algebra and let $\mathcal{C}_1, \ldots, \mathcal{C}_k$ be a family
of quantum chains which spans $\mathbb{C}^n$. Then there is a partition into
quantum chains $\mathcal{C}_1', \ldots, \mathcal{C}_{k'}'$ with $k' \leq k$
and each $\mathcal{C}_i'$ contained in some $\mathcal{C}_j$.
\end{prop}

(We might have $k' < k$ if some quantum chains disappear entirely in the
pruning process.)

Now we can prove the ``quantum'' Mirsky's theorem. Define
$\mathcal{A}^i(\mathbb{C}^n)$ to be the span of $\{A_i\cdots A_1v:
A_1, \ldots, A_i \in \mathcal{A}$, $v \in \mathbb{C}^n\}$, and set
$\mathcal{A}^0(\mathbb{C}^n) = \mathbb{C}^n$. If $k$ is the smallest value for
which $\mathcal{A}^k = \{0\}$, then the orthogonal differences
$$\mathcal{A}^{k-1}(\mathbb{C}^n) \oplus
(\mathcal{A}^{k-2}(\mathbb{C}^n) \ominus \mathcal{A}^{k-1}(\mathbb{C}^n))
\oplus\cdots\oplus
(\mathcal{A}^0(\mathbb{C}^n) \ominus \mathcal{A}^1(\mathbb{C}^n))$$
form an ordered partition into quantum antichains. I will call this the
{\it top down partition}. (There is also a {\it bottom up} partition
$\mathcal{E}_1 \oplus (\mathcal{E}_2\ominus \mathcal{E}_1) \oplus \cdots
\oplus (\mathcal{E}_k \ominus \mathcal{E}_{k-1})$ where $\mathcal{E}_1 =
\{v \in \mathbb{C}^n: Av = 0$ for all $A \in \mathcal{A}\}$ and inductively
$\mathcal{E}_{i+1} = \{v \in \mathbb{C}^n: Av \in \mathcal{E}_i$ for all
$A \in \mathcal{A}\}$. But we will not need this.)

\begin{theo}\label{qmir}
Let $\mathcal{A} \subset M_n$ be a nilpotent operator algebra. Then the
maximal length of a quantum chain equals the minimal size of an ordered
partition into quantum antichains.
\end{theo}

\begin{proof}
Fix an ordered partition $\mathcal{E}_1 \oplus \cdots \oplus \mathcal{E}_k$
of $\mathbb{C}^n$ into quantum antichains, with $k$ minimal.
Given any quantum chain $(v_1, \ldots, v_l)$, we have
$v_1 \in \mathcal{E}_1 \oplus \cdots \oplus \mathcal{E}_k = \mathbb{C}^n$,
and inductively, since $v_{j+1} \in \mathcal{A}\, v_j$ for all $j$,
and $\mathcal{A}(\mathcal{E}_j) \subseteq \mathcal{E}_1 \oplus
\cdots \oplus \mathcal{E}_{j-1}$ for all $j$, we
have $v_j \in \mathcal{E}_1 \oplus \cdots \oplus \mathcal{E}_{k + 1 - j}$.
This implies that the length of the quantum chain is at most $k$. We have
shown that no quantum chain has length greater than $k$.

For the converse, let $r$ be the smallest value for which $\mathcal{A}^r
= \{0\}$ and let $\mathcal{A}^{r-1}(\mathbb{C}^n) \oplus
(\mathcal{A}^{r-2}(\mathbb{C}^n) \ominus \mathcal{A}^{r-1}(\mathbb{C}^n))
\oplus\cdots\oplus
(\mathcal{A}^0(\mathbb{C}^n) \ominus \mathcal{A}^1(\mathbb{C}^n))$
be the top down partition into quantum antichains. Its size is $r$.
Since $\mathcal{A}^{r-1} \neq \{0\}$ there exists $v \in \mathbb{C}^n$ and
$A_1, \ldots, A_{r-1} \in \mathcal{A}$ such that $A_{r-1}\cdots A_1v \neq 0$.
Thus $(v, A_1v, \ldots, A_{r-1}\cdots A_1v)$ is a quantum chain of length
$r$. This shows that the maximal length of a quantum chain is at least as
large as the minimal size of an ordered partition into quantum antichains.
\end{proof}

Incidentally, this proof shows that the top down partition has minimal size
among all ordered partitions. As essentially the same proof would work with
the bottom up partition, it too has minimal size.

The proof of Theorem \ref{qmir} corresponds to an easy proof of the
classical theorem of Mirsky. Given any finite poset, let $\mathcal{C}_1$
be the set of maximal elements, let $\mathcal{C}_2$ be the set of new
maximal elements after $\mathcal{C}_1$ is removed, and so on. This yields
a partition into antichains such that every element of $\mathcal{C}_{i+1}$
lies under some element of $\mathcal{C}_i$, and one can then build a chain
whose length is the size of this partition
by starting with any element of the bottommost antichain and working up.

I mentioned earlier that in the quantum setting, arbitrary partitions into
quantum antichains cannot necessarily be converted into ordered partitions.
In fact, Theorem \ref{qmir} fails for unordered partitions.

\begin{exam}\label{nomir}
Let $\mathcal{A}$ be the subalgebra of $M_8$ generated by the matrices
$E_{31} + E_{61}$, $E_{23} + E_{73} - E_{26} + E_{46}$, and
$E_{57} + E_{84} - E_{52} + E_{82}$. Thus, it is the linear span of these
matrices, together with the matrices $E_{41} + E_{71}$, $E_{51} + E_{81}$,
and $E_{56} + E_{83}$. This is a nilpotent algebra, and
$\mathcal{E}_1 = {\rm span}\{e_1, e_2\}$,
$\mathcal{E}_2 = {\rm span}\{e_3, e_4, e_5\}$, and
$\mathcal{E}_3 = {\rm span}\{e_6, e_7, e_8\}$ are all quantum antichains for
$\mathcal{A}$. Thus $\mathcal{E}_1 \oplus \mathcal{E}_2 \oplus \mathcal{E}_3$
is a partition of $\mathbb{C}^8$ into three quantum antichains. But there is
a quantum chain of length 4, namely $(e_1, e_3 + e_6, e_4 + e_7, e_5 + e_8)$.
\end{exam}

The in some sense ``dual'' result to Mirsky's theorem is Dilworth's theorem,
which states that the minimal size of a partition into chains equals the
maximal width of an antichain. Surpisingly, the quantum version of the
trivial direction of this result fails.

\begin{exam}
Let $\mathcal{A} = {\rm span}\{E_{14}, E_{24}, E_{34}\} \subset M_4$. This is
a nilpotent operator algebra. It has a three-dimensional quantum antichain,
namely ${\rm span}\{e_1, e_2, e_3\}$, but it also has a partition into the two
quantum chains $(e_4, e_1)$ and $(e_4 + e_3, e_2)$.
\end{exam}

The harder direction of Dilworth's theorem does hold in the matrix
setting, however. None of the usual proofs successfully transfers to the
matrix setting, but there is a fairly easy linear algebra proof.

\begin{theo}\label{qdil}
Let $\mathcal{A} \subset M_n$ be a nilpotent operator algebra. Then the
minimal size of a partition of $\mathbb{C}^n$ into quantum chains is no
larger than the maximal width of a quantum antichain.
\end{theo}

\begin{proof}
Let $k$ be the largest value such that $\mathcal{A}^k$ is nonzero and let
$\mathcal{E}_i = \mathcal{A}^{i}(\mathbb{C}^n) \ominus
\mathcal{A}^{i+1}(\mathbb{C}^n))$ for $0 \leq i \leq k$,
so that $\mathcal{E}_k \oplus\cdots \oplus \mathcal{E}_0$ is the top down
partition. It will be convenient to have the indices descend in this way.
Let $d$ be the largest of the dimensions of the $\mathcal{E}_i$. We
will find $d$ quantum chains $(v_1, A^1_1v_1, \ldots, A^1_k\cdots A^1_1v_1)$,
$\ldots$, $(v_d, A^d_1v_d, \ldots, A^d_k\cdots A^d_1v_d)$ which span
$\mathbb{C}^n$. By Proposition \ref{chainspan} this is enough.

Let $d_i$ be the dimension of $\mathcal{E}_i$, for $0 \leq i \leq k$.
The goal will be to ensure that the vectors $v_1, \ldots, v_{d_0}$
orthogonally project to a basis of $\mathcal{E}_0$ and, for
$1 \leq i \leq k$, the vectors
$A^1_i\cdots A^1_1v_1$, $\ldots$, $A^{d_i}_i\cdots A^{d_i}_1v_{d_i}$
(the $(i+1)$st vectors in the first $d_i$ quantum chains, which live in
$\mathcal{E}_k \oplus \cdots \oplus \mathcal{E}_i$) orthogonally
project to a basis of $\mathcal{E}_i$. This will suffice because
it immediately implies that the terminal vectors in all the chains span
$\mathcal{E}_k$, and then inductively that the last $i$ vectors in all the
chains span $\mathcal{E}_k \oplus \cdots \oplus \mathcal{E}_{k+1-i}$ for each
$i$. Thus all the vectors in all the chains span $\mathbb{C}^n$.

The construction will be recursive, so that we choose the first $i$ vectors
in each chain before choosing any of the $(i+1)$st vectors. Those
$(i+1)$st vectors will themselves be chosen sequentially.
The key point is that if the vectors
$A^1_i\cdots A^1_1v_1$, $\ldots$, $A^{d_i}_i\cdots A^{d_i}_1v_{d_i}$
orthogonally project to a basis of $\mathcal{E}_i$, then the same will be
true of the vectors $\tilde{A}^1_i\cdots \tilde{A}^1_1\tilde{v}^1$,
$\ldots$, $\tilde{A}^{d_i}_i\cdots \tilde{A}^{d_i}_1\tilde{v}^{d_i}$ for any
$\tilde{A}^r_s$ sufficiently close to $A^r_s$ and $\tilde{v}^r$ sufficiently
close to $v_r$. This means that previous choices can be modified without
affecting the fact that their projections in $\mathcal{E}_j$ span
$\mathcal{E}_j$ for $j < i$, provided the modifications are sufficiently
small.

We can start by letting $v_1, \ldots, v_{d_0}$ be a basis of $\mathcal{E}_0$
and setting $v_i = 0$ for $i > d_0$. Having chosen the first $i$ vectors
in each of the chains, we aim to
choose the $(i+1)$st vectors sequentially, ensuring
that for each $j \leq d_i$ the $(i+1)$st elements of the first $j$ chains
project to a linearly independent set in $\mathcal{E}_i$. When choosing the
$(i+1)$st element of the $j$th chain, i.e., when choosing the operator
$A^j_i$ and possibly making small modifications to $v_j$ and to $A^j_1$,
$\ldots$, $A^j_{i-1}$, we just have to ensure that
the projection of $A^j_i\cdots A^j_1v_j$ into $\mathcal{E}_i$ does not lie
in a certain subspace, namely the span of the projections of the vectors
$A^{j'}_i\cdots A^{j'}_1v_{j'}$ into $\mathcal{E}_i$ for $j' < j$. Call this
span $\mathcal{F}$. To do this, find $w \in \mathbb{C}^n$ and
$B_1, \ldots, B_i \in \mathcal{A}$ such that the projection of
$B_i\cdots B_1w$ into $\mathcal{E}_i$ does not lie in $\mathcal{F}$. This
can be done because $\mathcal{A}^i(\mathbb{C}^n) = \mathcal{E}_k \oplus
\cdots \oplus \mathcal{E}_i$. It will then suffice to show that we can find
arbitrarily small values of $t$ such that the projection of
$$tB_i(A^j_{i-1} + tB_{i-1})\cdots (A^j_1 + tB_1)(v_j + tw)$$
into $\mathcal{E}_i$ does not lie in $\mathcal{F}$. Then we can define
$A^j_i$ to be $tB_i$ and replace $v_j$
with $v_j + tw$, $A^j_1$ with $A^j_1 + tB_1$, etc, and if $t$ is small enough
this will not affect the spanning property at previous stages.

Now if the projection of
$tB_i(A^j_{i-1} + tB_{i-1})\cdots (A^j_1 + tB_1)(v_j + tw)$
into $\mathcal{E}_i$ lies in $\mathcal{F}$ for all sufficiently small $t$,
then all of its derivatives at $t = 0$ must lie in $\mathcal{F}$. But the
$(i+1)$st derivative is $(i+1)B_i\cdots B_1w$, which does not lie in
$\mathcal{F}$, so this is impossible. Thus we are able to find arbitrarily
small values of $t$ which have the desired property.
\end{proof}

\section{Infinite dimensional examples}

Now we turn to the infinite dimensional setting. In infinite dimensions it
is natural to consider operator algebras which are weak* closed in
$B(\mathcal{H})$. These are called {\it dual operator algebras}.
In order to stay within this category, we must slightly modify the
definitions of subobject, quotient, and subquotient used in finite dimensions.

\begin{defi}\label{ihadef}
Let $\mathcal{A} \subseteq B(\mathcal{H})$ be a dual operator algebra and
let $P \in B(\mathcal{H})$ be the othogonal projection onto a closed subspace
$\mathcal{E} \subseteq \mathcal{H}$. Then $\overline{P\mathcal{A}P}^{wk*}$ is
\smallskip

{\narrower{
\noindent (i) a {\it subobject} of $\mathcal{A}$ if $\mathcal{E}$ is
invariant for $\mathcal{A}$;

\noindent (ii) a {\it quotient} of $\mathcal{A}$ if $\mathcal{E}$ is
coinvariant for $\mathcal{A}$;

\noindent (iii) a {\it subquotient} of $\mathcal{A}$ if $\mathcal{E}$ is
semi-invariant for $\mathcal{A}$.
\smallskip

}}
\noindent $\mathcal{A}$ is {\it hereditarily antisymmetric} if every
subquotient of $\mathcal{A}$ is antisymmetric.
\end{defi}

Of course, this definition reduces to Definition \ref{hadef} in the finite
dimensional setting, where weak* considerations become vacuous.

(The compression $P\mathcal{A}P$ is not automatically weak* closed. For
example, let $(x_n)$ be a dense sequence in the open unit disk $\mathbb{D}$
and define $x_0 = 0$ and
$x_{-k} = \frac{1}{k}$ for all $k \in \mathbb{N}$. Then the set $\mathcal{A}$
of sequences $(a_n)$ in $l^\infty(\mathbb{Z})$ with the property that
$a_0 = 0$ and the map
$x_n \mapsto a_n$ extends to a bounded analytic function on $\mathbb{D}$ is a
weak* closed algebra --- see the proof of Proposition \ref{hinfty} below ---
which by the Schwarz lemma satisfies $|a_{-k}| \leq
\frac{1}{k} \sup|a_n|$. The compression of $\mathcal{A}$ to
$l^\infty(-\mathbb{N})$ is therefore contained in $c_0$ and contains all
finite sequences, so it is not weak* closed. In this example
$l^\infty(-\mathbb{N})$ is both invariant and coinvariant. The algebra
$\mathcal{A}$ is not unital, but its unitization has the same property that
its compression to $l^\infty(-\mathbb{N})$ is not weak* closed.)

The basic results from Section 2 go through without significant modification
in infinite dimensions.

\begin{prop}\label{icompofcomp}
Let $\mathcal{A} \subseteq B(\mathcal{H})$ be a dual operator algebra. Then
any subquotient of a subquotient of
$\mathcal{A}$ is a subquotient of $\mathcal{A}$.
\end{prop}

This works because an operator algebra and its weak* closure have the same
invariant subspaces, and hence the same semi-invariant subspaces, and because
the compression of any weak* convergent net is weak* convergent.

\begin{coro}\label{isubha}
Any subquotient of a hereditarily antisymmetric dual operator
algebra is hereditarily antisymmetric.
\end{coro}

In infinite dimensions we generalize Definition \ref{compdef} by taking a
companion subspace $\mathcal{F}$ to be any topological complement of
$\mathcal{E}_2$ in $\mathcal{E}_1$. Thus, it is a closed subspace of
$\mathcal{E}_1$ satisfying $\mathcal{E}_2 + \mathcal{F} = \mathcal{E}_1$
and $\mathcal{E}_2 \cap \mathcal{F} = \{0\}$.

\begin{prop}\label{icompprop}
Let $\mathcal{E} = \mathcal{E}_1\ominus\mathcal{E}_2$ be a semi-invariant
subspace for a dual operator algebra $\mathcal{A} \subseteq B(\mathcal{H})$
and let $\mathcal{F}$ be a companion subspace of $\mathcal{E}$. Let $P$ be the
orthogonal projection onto $\mathcal{E}$, let $P_0: \mathcal{F} \to
\mathcal{E}$ be its restriction to $\mathcal{F}$, and let
$Q \in M_n$ be the natural projection onto $\mathcal{F}$. Then
$\Phi: T \mapsto P_0TP_0^{-1}$ defines an isomorphism between
$\overline{Q\mathcal{A}Q}^{wk*} \subseteq B(\mathcal{F})$ and
$\overline{P\mathcal{A}P}^{wk*} \subseteq B(\mathcal{E})$.
\end{prop}

The proof of this result requires the one additional observation that
$P_0$ is invertible with bounded inverse by the Banach isomorphism theorem.

Propositions \ref{sainf} and \ref{unitprop} were already stated for possibly
infinite dimensional operator algebras. The infinite dimensional analog of
Proposition \ref{tfaefd} fails, however.

\begin{exam}\label{inormal}
Let $\mathcal{H} = l^2(\mathbb{Z})$ and let $\mathcal{A}$ be the set of
operators in $\mathcal{B}(\mathcal{H})$ whose matrices relative to the
standard basis
of $l^2(\mathbb{Z})$ are upper triangular and constant on every diagonal.
That is, the operators in $\mathcal{A}$ satisfy $\langle Ae_j,e_i\rangle = 0$
when $i > j$ and $\langle Ae_j,e_i\rangle = \langle Ae_{j+1},e_{i+1}\rangle$
for all $i, j \in \mathbb{Z}$.

This is a unital dual operator algebra, and it is antisymmetric
because if $A \in \mathcal{A}$ is self-adjoint then $\langle Ae_j,e_i\rangle
= 0$ for all $i > j$ implies $\langle Ae_j, e_i\rangle = 0$ for all $i < j$,
i.e., $A$ is diagonal and hence a scalar multiple of $I$. But $\mathcal{A}$
contains the bilateral shift operator $U: e_i \mapsto e_{i-1}$, which is
normal and even unitary but not a scalar multiple of the identity. (In fact,
$\mathcal{A}$ is the unital dual operator algebra generated by $U$.)
\end{exam}

Let us look at some infinite dimensional examples of hereditarily
antisymmetric operator algebras. First we generalize the algebras
$\mathcal{T}_n$ of Example \ref{ex1} to infinite dimensions. There are
a variety of ways to do this.

\begin{exam}\label{iex1}
Given a totally ordered set $(X, \preceq)$, define
$\mathcal{T}_X \in B(l^2(X))$ to be the set of operators whose matrix
relative to the standard basis $\{e_x: x \in X\}$ is upper triangular and
constant on the main diagonal. That is, $\langle Ae_y, e_x\rangle = 0$
whenever $y \prec x$ and $\langle Ae_x,e_x\rangle = \langle Ae_y,e_y\rangle$
for all $x,y \in X$.
\end{exam}

\begin{prop}\label{txprop}
For any totally ordered set $(X,\preceq)$ the algebra $\mathcal{T}_X$ is
unital, weak* closed, and hereditarily antisymmetric.
\end{prop}

\begin{proof}
Closure under weak* limits is seen by by examining matrix entries.
Hereditary antisymmetry follows from identifying the semi-invariant subpaces
of $\mathcal{T}_X$ as those of the form ${\rm span}\{e_x: x \in I\}$ where
$I$ is an interval in $X$. The compression of $\mathcal{T}_X$ to such a
subspace would simply be $\mathcal{T}_I$, which is still antisymmetric
(and already weak* closed).
\end{proof}

The simplest cases are $X = \mathbb{Z}$, $\mathbb{N}$, and $-\mathbb{N}$
(the negative integers, or equivalently $\mathbb{N}$ with the order
reversed). $\mathcal{T}_{-\mathbb{N}}$ could equivalently be defined to be
the set of bounded operators on $\mathbb{N}$ whose matrix for the standard
basis is lower triangular and constant on the main diagonal. Using the
opposite order on $\mathbb{N}$ effectively interchanges upper and lower
triangular matrices.
Note that $\mathcal{T}_{\mathbb{N}}$ and $\mathcal{T}_{-\mathbb{N}}$ are
different: the former has a minimal invariant subspace which is the range of
the operator $E_{12}$, while the latter has no minimal invariant subspace.

\begin{prop}
$\mathcal{T}_{\mathbb{Z}}$, $\mathcal{T}_{\mathbb{N}}$, and
$\mathcal{T}_{-\mathbb{N}}$ are maximal hereditarily antisymmetric algebras.
\end{prop}

\begin{proof}
They are hereditarily antisymmetric by Proposition \ref{txprop}.
For maximality, Let $\mathcal{A}$ be an operator algebra which properly
contains one of these algebras. There are four cases. First, if the matrix
of some operator $A \in \mathcal{A}$ has a nonzero entry $a_{ij}$ with
$i \geq j + 2$, then $E_{i-1,i}, E_{j,i-1} \in \mathcal{A}$ and
$E_{i-1,i}AE_{j,i-1}$ is a nonzero multiple of $E_{i-1,i-1}$. So $\mathcal{A}$
is not even antisymmetric. Second, if some $A \in \mathcal{A}$ has no
nonzero matrix entries more than one diagonal below the main diagonal,
but adjacent nonzero entries $a_{i+1,i}$, $a_{i,i-1}$ on the first
subdiagonal, then $A^2$ has a nonzero $(i+1,i-1)$ entry, reducing to the
first case. Third, if no operator in $\mathcal{A}$ has
nonzero matrix entries more than one diagonal below the main diagonal,
but some $A \in \mathcal{A}$ has a nonzero entry $a_{i+1,i}$, then both
$a_{i+2,i+1}$ and $a_{i,i-1}$ must be zero (if they both exist;
in the $\mathbb{N}$ or $-\mathbb{N}$ settings one of these entries could
be out of range). Moreover, those same entries must be zero for any
$B \in \mathcal{A}$, as otherwise a linear combination of $A$ and $B$
would put us in the second case. It follows that ${\rm span}\{\ldots,
e_{i-1}, e_i, e_{i+1}\}$ and ${\rm span}\{\ldots, e_{i-1}\}$ are both
invariant, and so their orthogonal difference ${\rm span}\{e_i,e_{i+1}\}$ is
semi-invariant. The compression of $\mathcal{A}$ to this two-dimensional
subspace contains the identity matrix, the matrix $E_{i,i+1}$, and the
compression of $A$, which is not upper triangular. So it is at least
three-dimensional and therefore not antisymmetric by Proposition
\ref{dimcount}.

In the final case, every operator in $\mathcal{A}$ is upper triangular but
$\mathcal{A}$ includes an operator $A$ whose main diagonal entries are not
constant. By subtracting a strictly upper triangular operator, we can assume
that $A$ is diagonal. Say $\langle Ae_i,e_i\rangle \neq
\langle Ae_{i+1}, e_{i+1}\rangle$. Then ${\rm span}\{e_i,e_{i+1}\}$ is
semi-invariant, and the compression of $A$ to this subspace
is diagonal but not a scalar multiple of the identity. So it is a non-scalar
normal operator, and this shows that the compression of $\mathcal{A}$ is not
antisymmetric by Proposition \ref{tfaefd}. We have shown that no operator
algebra which properly contains $\mathcal{T}_{\mathbb{Z}}$,
$\mathcal{T}_{\mathbb{N}}$, or $\mathcal{T}_{-\mathbb{N}}$ is hereditarily
antisymmetric.
\end{proof}

It was easier to show that the algebras $\mathcal{T}_n$ are maximal
hereditarily antisymmetric, because these algebras were even maximal
antisymmetric, which is an easier condition to check. However, that fact
relied on Proposition \ref{tfaefd}, which no longer holds in infinite
dimensions. In fact $\mathcal{T}_X$ is never maximal antisymmetric if
$X$ is infinite.

\begin{prop}\label{hinfty}
$l^\infty$ contains an infinite dimensional, weak* closed, unital,
antisymmetric subalgebra.
\end{prop}

\begin{proof}
Let $(x_n)$ be a dense sequence in the open unit disk $\mathbb{D}$. Define
$\mathcal{A} \subset l^\infty$ to be the set of sequences $(a_n)$ with the
property that the map $x_n \mapsto a_n$ extends to a bounded analytic
function on $\mathbb{D}$. This is clearly an infinite dimensional unital
algebra, and
it is antisymmetric because any analytic function which takes real values
on a dense subset of $\mathbb{D}$ must be constant. For weak* closure, by
the Krein-Smulian theorem it suffices to check closure under bounded
pointwise convergence; since the predual of $l^\infty$ is separable, it
suffices to consider bounded pointwise convergent sequences. If $(f_n)$ is
a sequence of analytic functions on $\mathbb{D}$ whose restrictions to the
set $\{x_n\}$ are uniformly bounded and converge pointwise, then by continuity
the $f_n$ must be uniformly bounded, and Vitali's theorem (a consequence of
Montel's theorem) then implies that this sequence converges uniformly on
compact sets to a bounded analytic function on $\mathbb{D}$. So the pointwise
limit of the restrictions to the set $\{x_n\}$ still belongs to $\mathcal{A}$.
\end{proof}

\begin{prop}\label{maxant}
If $(X,\preceq)$ is any infinite totally ordered set, then
$\mathcal{T}_X$ is properly contained in another weak* closed antisymmetric
algebra.
\end{prop}

\begin{proof}
Fix a surjection $\phi: X \to \mathbb{N}$, let $\mathcal{A}$ be as in
Proposition \ref{hinfty}, and define $\mathcal{B} \subset B(l^2(X))$
to be the set of all operators whose matrix relative to the standard
basis is upper triangular and whose main diagonal entries equal
$f\circ \phi$ for some $f \in \mathcal{A}$. One straightforwardly checks
that $\mathcal{B}$ is a weak* closed antisymmetric algebra.
\end{proof}

If $X$ is not discretely ordered, worse things can happen.

\begin{prop}
$\mathcal{T}_{\mathbb{Q}}$ is properly contained in another weak* closed
hereditarily antisymmetric algebra.
\end{prop}

\begin{proof}
First, write $\mathbb{Q}$ as the disjoint union of a sequence of subsets
$X_n$ each of which is dense in $\mathbb{Q}$. (For instance, $X_n$ could
be the set of all rationals which when written in lowest terms have a
denominator whose smallest prime factor is the $n$th prime, including
$\mathbb{Z}$ in $X_1$, say.) Define $\phi: \mathbb{Q} \to \mathbb{N}$ by
setting $\phi(x) = n$ when $x \in X_n$ and let $\mathcal{B}$ be as in the
proof of Proposition \ref{maxant}, for this $\phi$.

In this case, $\mathcal{B}$ is
hereditarily antisymmetric because the semi-invariant subspaces are
precisely the subspaces of the form ${\rm span}\{e_x: x \in I\}$ for some
interval $I$ in $\mathbb{Q}$, and the compression of $\mathcal{B}$ to any
such subspace, if $I$ contains more than a single point, consists of upper
triangular operators whose diagonal entries take all the values of some
function in $\mathcal{A}$. This uses the fact that every interval in
$\mathbb{Q}$ of positive length contains points
from every $X_n$. Since $\mathcal{A}$ is antisymmetric,
the diagonal entries of any such compression, if nonconstant, cannot
all be real, showing that the only self-adjoint operators in the compression
of $\mathcal{B}$ are scalar multiples of the identity. We have shown that
$\mathcal{B}$ is hereditarily antisymmetric.
\end{proof}

We may also consider continuous analogs of $\mathcal{T}_X$.

\begin{defi}\label{aepsilon}
For each $\epsilon > 0$ define $\mathcal{T}^0_\epsilon$ to be the
set of operators $A \in B(L^2(\mathbb{R}))$ which satisfy
$\langle Af,g\rangle = 0$ whenever $f$ is supported on
$(-\infty, a + \epsilon]$ and $g$ is supported on $[a,\infty)$, for
some $a \in \mathbb{R}$. Equivalently, $A$ takes $L^2((-\infty,a + \epsilon])
\subset L^2(\mathbb{R})$ into $L^2((-\infty, a])$ for each $a$. This is a
dual operator algebra. Let $\mathcal{T}_\epsilon$ be the unitization of
$\mathcal{T}_\epsilon^0$.
\end{defi}

In regard to the next result, note that any invariant subspace for the
union of a chain $\{A_\lambda\}$ of operator algebras is invariant
for each $A_\lambda$. So the same is true of semi-invariant
subspaces, and this means that if each $A_\lambda$ is hereditarily
antisymmetric then $\bigcup A_\lambda$ cannot contain any operators which
compress to a nonscalar self-adjoint operator on some semi-invariant
subspace. However, its weak* closure might.

\begin{prop}\label{aepsilonprop}
For each $\epsilon > 0$ the algebra $\mathcal{T}_\epsilon$ is hereditarily
antisymmetric, but the weak* closure of the union
$\bigcup_{\epsilon > 0} \mathcal{T}_\epsilon$ is not even antisymmetric.
\end{prop}

\begin{proof}
The nontrivial invariant subspaces for $\mathcal{T}_\epsilon^0$
(or $\mathcal{T}_\epsilon$) are precisely
the subspaces of the form $L^2((-\infty, a] \cup X) \subset L^2(\mathbb{R})$
where $a \in \mathbb{R}$ and $X$ is a measurable subset of $[a,a + \epsilon]$.
It follows that every semi-invariant subspace has the form $L^2(Y)$ for some
measurable $Y \subseteq \mathbb{R}$ (which is the difference of two sets of
the preceding form, but we do not need this).

Let $Y$ be any measurable subset of $\mathbb{R}$ and let $P$ be the
orthogonal projection of $L^2(\mathbb{R})$ onto $L^2(Y)$. Then for any
$A \in \mathcal{T}_\epsilon^0$, the compression $PAP$ satisfies
$\langle PAP f,g\rangle = \langle Af,g\rangle = 0$ whenever $f$ is supported on
$(-\infty, a+\epsilon] \cap Y$ and $g$ is supported on $[a,\infty) \cap Y$,
for some $a$. Thus the same is true of any operator in
$\overline{P\mathcal{T}^0_\epsilon P}^{wk*}$. So any self-adjoint operator
$B$ in this set must satisfy $\langle Bf,g\rangle = \langle Bg,f\rangle = 0$
for all such $f$ and $g$. In particular, if $f$ is supported on
$[a,a + \epsilon]$ then $\langle Bf,g\rangle = 0$ if $g$ is either supported
on $[a,\infty)$ or on $(-\infty, a + \epsilon]$, which implies that $Bf = 0$.
This implies that $B = 0$. This shows that every subquotient of
$\mathcal{T}^0_\epsilon$ is antisymmetric, i.e., $\mathcal{T}^0_\epsilon$ is
hereditarily antisymmetric. Hereditary antisymmetry of $\mathcal{T}_\epsilon$
follows from Proposition \ref{unitprop}.

For every $f \in L^\infty(\mathbb{R})$ and $r > 0$ the union
$\bigcup_{\epsilon > 0} \mathcal{T}_\epsilon$ contains the operator
$(Ag)(x) = f(x)g(x + r)$ which shifts everything in $L^2(\mathbb{R})$ left
by $r$ and then multiplies by $f$. As $r \to 0$ these operators converge
weak* to the operator of multiplication by $f$. So every multiplication
operator belongs to
$\overline{\bigcup_{\epsilon > 0} \mathcal{T}_\epsilon}^{wk*}$, and thus
this algebra is not antisymmetric.
\end{proof}

This example shows that the restriction to weak* closed algebras is
important. We can have an algebra whose compression to any semi-invariant
subspace contains no nonscalar self-adjoint operators, but whose weak*
closure does not have this property (indeed, whose weak* closure itself
contains nonscalar self-adjoint operators).

The phenomenon exhibited in Proposition \ref{aepsilonprop} unfortunately
limits our ability to reduce the analysis of arbitrary
hereditarily antisymmetric dual operator algebras to the analysis of
maximal hereditarily antisymmetric dual operator algebras.

There is also a natural generalization of Example \ref{exam2} to infinite
dimensions.

\begin{exam}\label{iexam2}
Let $\mathcal{P} \subset B(\mathcal{H})$ be a maximal family of commuting
(nonorthogonal) projections. Then its commutant
$$\mathcal{P}' = \{A \in B(H): AP = PA\mbox{ for all }P \in \mathcal{P}\}$$
is a unital dual operator algebra.
\end{exam}

The following special case is important enough to merit a mention.

\begin{prop}\label{schauder}
Suppose $(v_n)$ is a Schauder basis of the Hilbert space $\mathcal{H}$.
Then the projections $P_n$ satisfying $P_nv_n = v_n$ and $P_nv_k = 0$
for $k \neq n$ are uniformly bounded and commute. There is exactly one
maximal family of commuting projections which contains the set $\{P_n\}$,
and it consists of precisely those operators $P_X$, for some
$X \in \mathbb{N}$, which are bounded and satisfy $P_Xv_n = v_n$ if
$n \in X$ and $P_Xv_n = 0$ if $n \not\in X$.
\end{prop}

The first assertion is a standard fact about Schauder bases, and the second
is just the easy observation that any projection that commutes with every
$P_n$ must have the stated form.

I will call the maximal family of commuting projections identified in
Proposition \ref{schauder} {\it the family of projections associated
with the Schauder basis $(v_n)$}.

\begin{prop}
Let $\mathcal{P} \subset B(\mathcal{H})$ be a maximal family of commuting
projections. Then $\mathcal{P}'$ is antisymmetric if and only if
$\mathcal{P}$ contains no orthogonal projections besides $0$ and $I$.
\end{prop}

\begin{proof}
By maximality, there are no projections in $\mathcal{P}'$ other than those
in $\mathcal{P}$. The result therefore follows from Proposition \ref{sainf}.
\end{proof}

In order to characterize hereditary antisymmetry, we need a slightly stronger
hypothesis.

\begin{theo}\label{iha}
Let $\mathcal{P} \subset B(\mathcal{H})$ be a maximal family of commuting
projections which is uniformly bounded. Then $\mathcal{P}'$ is hereditarily
antisymmetric if and only if whenever $P, Q, R \in \mathcal{P}$ satisfy
$PQ = PR = QR = 0$ but $P, Q \neq 0$, the orthogonal projections of
${\rm ran}(P)$ and ${\rm ran}(Q)$ into ${\rm ran}(R)^\perp$ are not
mutually orthogonal.
\end{theo}

\begin{proof}
Suppose $P, Q, R \in \mathcal{P}$ have the stated properties. Let $\tilde{R}$
be the orthogonal projection onto ${\rm ran}(R)^\perp$. Since
${\rm ran}(R)^\perp$ is coinvariant and therefore
semi-invariant, the map $A \mapsto \tilde{R}A\tilde{R}$ is a homomorphism,
and thus $\tilde{P} = \tilde{R}P\tilde{R}$ and $\tilde{Q} =
\tilde{R}Q\tilde{R}$ are commuting projections. This shows that
$\tilde{R}\mathcal{P}'\tilde{R}$ contains a pair of projections whose
product is zero and whose ranges are orthogonal, and from this we get that
the range of $\tilde{P} + \tilde{Q}$ is semi-invariant (cf.\ Proposition
\ref{compofcomp}) and the compressions of $P$ and $Q$ to this subspace
are nonzero orthogonal projections which sum to the identity. So
$\mathcal{P}'$ has a subquotient which contains a nonscalar orthogonal
projection, i.e., it is not hereditarily antisymmetric.

For the converse, suppose $\mathcal{P}'$ is not hereditarily antisymmetric.
The set $\mathcal{U} = \{2P - I: P \in \mathcal{P}\}$
is an abelian group under operator product, so, using the uniform boundedness
hypothesis, a theorem of Day and Dixmier \cite{Da, Di} implies that it can
be conjugated to a family of unitaries. That is, there exists an invertible
$S \in B(\mathcal{H})$ such that $S^{-1}(2P - I)S = 2S^{-1}PS - I$ is unitary
for all $P \in \mathcal{P}$. But the spectrum of each of these operators is
contained in $\{1, -1\}$, so each of these unitaries is self-adjoint and
has the form
$2\tilde{P} - I$ for some orthogonal projection $\tilde{P} = S^{-1}PS$. Thus
the set $S^{-1}\mathcal{P}S$ is a maximal commuting family of orthogonal
projections, i.e., it is the set of projections in the maximal abelian von
Neumann algebra $(S^{-1}\mathcal{P}S)'$.

Now $(S^{-1}\mathcal{P}S)' = S^{-1}\mathcal{P}'S$, so the invariant subspaces
for $\mathcal{P}'$ are precisely the subspaces of the form
$S(\mathcal{E})$ where $\mathcal{E}$ is invariant for
$(S^{-1}\mathcal{P}S)'$. But the invariant subspaces for the latter are
just the ranges of the projections in $S^{-1}\mathcal{P}S$. Thus we have
shown that the invariant subspaces for $\mathcal{P}$ are precisely the
ranges of the projections in $\mathcal{P}$. The semi-invariant subspaces
are therefore the orthogonal differences between ranges of projections in
$\mathcal{P}$.

Since $\mathcal{P}'$ is not hereditarily antisymmetric, there is a
semi-invariant subspace $\mathcal{E}$ such that the compression of
$\mathcal{P}'$ to $\mathcal{E}$ is not antisymmetric. Say
$\mathcal{E} = {\rm ran}(R_1) \ominus {\rm ran}(R_2)$ for some projections
$R_1, R_2 \in \mathcal{P}$ with ${\rm ran}(R_2) \subset {\rm ran}(R_1)$.

Now $R_0 = R_1 - R_2$ is the natural projection onto a companion subspace
of $\mathcal{E}$, and $R_0\mathcal{P}'R_0 = \mathcal{P}'R_0$ is weak* closed,
so by Proposition \ref{icompprop} so is the compression of $\mathcal{P}'$ to
$\mathcal{E}$.

Since this compression is not antisymmetric, it therefore contains
nonscalar orthogonal projections $P$ and $Q$ whose sum is the orthogonal
projection onto $\mathcal{E}$. These correspond via Proposition
\ref{icompprop} to projections $\tilde{P}, \tilde{Q} \in \mathcal{P}R_0
\subseteq \mathcal{P}$ whose ranges orthogonally project onto orthogonal
subspaces of ${\rm ran}(R_2)^\perp$.
\end{proof}

If $\mathcal{P}$ is the family of projections associated with some Schauder
basis, as in Proposition \ref{schauder}, and the basis is actually Riesz,
then we are in the setting of Theorem \ref{iha}. In this case, at least,
the operator algebra $\mathcal{P}'$ is maximal hereditarily antisymmetric.

\begin{theo}
Suppose $\mathcal{P}$ is the family of projections associated with a Riesz
basis $(v_n)$. If $\mathcal{P}'$ is hereditarily antisymmetric then it is
maximal hereditarily antisymmetric.
\end{theo}

\begin{proof}
Let $\mathcal{A}$ be a dual operator algebra which properly contains
$\mathcal{P}'$. Then there exists $k \in \mathbb{N}$ such that $v_k$ is not
an eigenvector for some operator in $\mathcal{A}$. Let $\{P_n\}$ be the
projections from Proposition \ref{schauder} and let $X$ be the set of
$n \in \mathbb{N}$ such that $P_n(Av_k) \neq 0$, for some
$A \in \mathcal{A}$. Observe that for any such $n$ we have $0 \neq P_nAP_k
\in \mathcal{A}$, and hence $E_{nk} \in \mathcal{A}$, where as before $E_{ij}$
is the operator which takes $v_j$ to $v_i$ and annihilates all other $v_{j'}$.

Now $\mathcal{E} = {\rm span}\{v_n: n \in X\}$ is an invariant subspace
for $\mathcal{A}$ which contains $v_k$ and at least one other $v_n$.
Working in $\mathcal{E}$, find a nonzero vector $v$ which is orthogonal to
${\rm span}\{v_n: n \in X \setminus \{k\}\}$ and write
$v = \sum_{n \in X} a_nv_n$. Then consider the operator
$A = \sum_{n \in X} a_nE_{nk}$; considered as an operator in $B(\mathcal{E})$,
this is a nonzero scalar multiple of the orthogonal projection onto
${\rm span}\{v\}$, and it belongs to $\mathcal{A}$ because each $E_{nk}$
belongs to $\mathcal{A}$ and the partial sums are uniformly bounded.
This shows that $\mathcal{A}$ is not hereditarily antisymmetric.
\end{proof}

\section{Infinite dimensional structure analysis}

The transitive algebra problem asks whether any dual operator algebra that
is properly contained in $B(\mathcal{H})$ must have a nontrivial invariant
subspace. Without knowing this to be the case,
there is little we can say about the
structure of such algebras. However, assuming the problem has a positive
answer, we easily get an infinite dimensional analog of Theorem \ref{utthm}.
As in the finite dimensional case, say that a subquotient of a dual operator
algebra corresponding to a semi-invariant subspace $\mathcal{E}$ is {\it full}
if it equals $B(\mathcal{E})$. We also need an infinite dimensional version
of upper triangularity.

\begin{defi}
A {\it nest} in a Hilbert space $\mathcal{H}$ is a chain of
closed subspaces, i.e., a family of closed subspaces
which is totally ordered by inclusion. It is {\it maximal} if it is not
properly contained in any other nest. An algebra $\mathcal{A} \subseteq
B(\mathcal{H})$ is {\it upper triangular} for a nest if each subspace in
the nest is invariant for $\mathcal{A}$.
\end{defi}

Note that in finite dimensions a maximal nest simply looks like a nested
sequence of subspaces, one of each possible dimension, and being upper
triangular with respect to a maximal nest is the same as being upper
triangular with respect to some orthonormal basis. In infinite dimensions,
a nest is maximal if and only if if contains $\{0\}$ and $\mathcal{H}$,
it is complete (closed under arbitrary joins and meets), and whenever
$\mathcal{E}_1$ and $\mathcal{E}_2$ are distinct subspaces in the nest, with
$\mathcal{E}_2 \subset \mathcal{E}_1$ and no other subspace in the nest
intermediate between them, then $\mathcal{E}_1$ has codimension $1$ in
$\mathcal{E}_2$.

\begin{theo}\label{istructure}
Let $\mathcal{A} \subseteq B(\mathcal{H})$ be a dual operator algebra with
no full subquotients of dimension greater than $1$. If the transitive algebra
problem has a positive solution, then there is a maximal nest in $\mathcal{H}$
with respect to which $\mathcal{A}$ is upper triangular.
\end{theo}

\begin{proof}
Use Zorn's lemma to find a maximal chain of invariant subspaces. We must
show that it is a maximal nest. It clearly contains $\{0\}$ and $\mathcal{H}$
and is complete. Thus let $\mathcal{E}_1$ and $\mathcal{E}_2$
be distinct subspaces in the chain and suppose $\mathcal{E}_2 \subset
\mathcal{E}_1$ and there is no strictly intermediate subspace between them in
the chain. Then $\mathcal{E}_1\ominus\mathcal{E}_2$ is semi-invariant, and if
its dimension were greater than $1$ then by hypothesis the corresponding
subquotient of $\mathcal{A}$ could not be full. A positive solution to the
transitive algebra problem would then imply the existence of a proper
closed subspace $\mathcal{F}$ of $\mathcal{E}_1\ominus \mathcal{E}_2$ which
is invariant for the compression of $\mathcal{A}$. But then $\mathcal{E}_2
\oplus \mathcal{F}$ would be an invariant subspace strictly intermediate
between
$\mathcal{E}_1$ and $\mathcal{E}_2$, contradicting maximality of the chain.
We conclude that there is a maximal nest which consists only of invariant
subspaces for $\mathcal{A}$, i.e., with respect to which $\mathcal{A}$ is
upper triangular.
\end{proof}

Like the implication (iii) $\Rightarrow$ (i) of Theorem \ref{utthm}, this
theorem can be inferred from standard results; see Lemma 7.1.11 of
\cite{RR}. But its explicit statement is perhaps new.

\begin{coro}
Let $\mathcal{A} \subseteq B(\mathcal{H})$ be a hereditarily antisymmetric
dual operator algebra. If the transitive algebra problem has a positive
solution, then there is a maximal nest in $\mathcal{H}$ with respect to
which $\mathcal{A}$ is upper triangular.
\end{coro}

The same technique can be applied to a single bounded operator. In this case
the hypothesis we need is a positive solution to the invariant subspace
problem. The following theorem is proven in the same way as Theorem
\ref{istructure} and hence also follows easily from ideas in \cite{RR}.
Experts would surely consider it to be ``known'',
but I have not seen it explicitly written anywhere.

\begin{theo}\label{ispthm}
Let $A$ be a bounded operator on an infinite dimensional Hilbert space.
Assume the invariant subspace problem for Hilbert space operators has a
positive solution. Then there is a maximal nest in $\mathcal{H}$ with
respect to which $A$ is upper triangular.
\end{theo}

It seems worthwhile to state this theorem explicitly, both to make the result
available to non-experts, and because it yields a reduction in the negative
direction: in order to answer the invariant subspace problem negatively, we
do not need to find a bounded operator for which every nonzero vector is
cyclic, we only need to find a bounded operator which cannot be made upper
triangular.

In contrast to the finite dimensional case (Theorem \ref{utthm}), the
converse direction in Theorem \ref{istructure} is false.

\begin{exam}\label{iutex}
Working on $l^2(\mathbb{N})$, let $A_1 = {\rm diag}(1,0,1,0,\ldots)$,
$A_2 = {\rm diag}(0,1,0,1,\ldots)$, $A_3 = U\cdot {\rm diag}(2,0,2,0,\ldots)$,
and $A_4 = U\cdot{\rm diag}(0,1,0,1,\ldots)$, where $U$ is the backward
shift operator. Consider the vectors
$$v = \left[\begin{matrix}
1\\
0\\
\frac{1}{2}\\
0\\
\frac{1}{4}\\
0\\
\vdots
\end{matrix}\right]
\qquad w = \left[\begin{matrix}
0\\
1\\
0\\
\frac{1}{2}\\
0\\
\frac{1}{4}\\
\vdots
\end{matrix}\right].$$
Then $A_1v = v$, $A_1w = 0$, $A_2v = 0$, $A_2w = w$, $A_3v = w$, $A_3w = 0$,
$A_4v = 0$, $A_4w = v$. Thus the algebra generated by the $A_i$ has
$\mathcal{E} = {\rm span}\{v,w\}$ as an invariant subspace, and its
compression to $\mathcal{E}$ equals $B(\mathcal{E}) \cong M_2$. So it has
a full subquotient of dimension greater than $1$, yet it is evidently
upper triangular for the standard orthonormal basis of $l^2(\mathbb{N})$.
\end{exam}

\end{document}